\documentclass{article}

\usepackage{amsmath} % 数学公式
\usepackage{amsthm}
\usepackage{amsfonts} %\textbf{}% 数学字体
\usepackage{amssymb} % 数学符号
\usepackage{graphicx} % 图片插入
\usepackage{hyperref} % 超链接
\usepackage{geometry} % 页面设置
\usepackage{float}
\usepackage{comment}
\usepackage{enumitem} % 列表设置
\usepackage{csquotes}
\usepackage{xcolor}
\newcommand{\keywords}[1]{\par\noindent\textbf{Keywords: }#1}

\newtheorem{proposition}{Proposition}
\newtheorem{theorem}{Theorem}
\newtheorem*{remark}{Remark}
\newtheorem{corollary}{Corollary}
\newtheorem{lemma}{Lemma}
\newtheorem{definition}{Definition}
\newtheorem{example}{Example}

\DeclareMathOperator{\Ev}{Ev}
\DeclareMathOperator{\wt}{wt}
\DeclareMathOperator{\im}{im}
\DeclareMathOperator{\Supp}{Supp}
\DeclareMathOperator{\Span}{span}
\DeclareMathOperator{\GL}{\mathrm{GL}}
\DeclareMathOperator{\PGL}{PGL}
\DeclareMathOperator{\Sym}{Sym}
\title{3-Designs from $\mathrm{GL}_2(\mathbb{F}_q)$-Invariant Subspaces of $\mathbb F_q[X,Y]_k$}
\date{}   % 或 \date{} 不显示日期
\author{Lewen Wang\and Huawei Wu\and Sihuang Hu}
\begin{document}
\maketitle
{\renewcommand{\thefootnote}{}
  		\footnotetext{
        This work was supported in part by the National Key Research and
Development Program of China under Grant 2021YFA1001000, in part by
the National Natural Science Foundation of China under Grant 12571576 and
Grant 12231014, and in part by the Shandong Provincial Natural Science
Foundation under Grant ZR2025QA05.

 			 Lewen Wang and Sihuang Hu are with the School of Cyber Science and Technology, Shandong University, Qingdao, Shandong 266237, China (emails: lewenwang3@gmail.com and husihuang@sdu.edu.cn). 
             Huawei Wu is with Shanghai Fintelli Box Technology Co., Ltd., Shanghai 200127, China (email: wuhuawei1996@gmail.com). The corresponding author is Sihuang Hu.
  		}
    }

%\begin{abstract}
%We develop a systematic framework for constructing $3$-designs from $\GL_2(\mathbb F_q)$-invariant subspaces of $\Sym^k(\mathbb F_q^2)$.
%Via the evaluation map on $\mathbb P^1(\mathbb F_q)$, these subspaces give rise to subcodes of the projective Reed--Solomon code. 
%From this perspective, both the supports of minimum-weight codewords in these subcodes and the supports of suitable fixed-weight codewords in their duals  yield $3$-designs. 
%Using the Cayley transform, we reformulate the construction on the unit circle and express the block conditions in terms of elementary symmetric polynomials.
%We then apply this framework to the Lucas subspaces $W_k^{\mathrm{Luc}}$, obtaining explicit block descriptions, nonemptiness results, and basic families of blocks.
%In particular, for $k=p^m+1$ we determine the resulting Steiner system associated with $W_k^{\mathrm{Luc}}$.
%In the ternary case $p=3$ and $k=7$, we connect the associated design with the ternary Melas code and explicitly determine the parameters left open in (Xu et al., Designs, Codes and Cryptography: Vol. 92, 2024).
%\end{abstract}

% \begin{keywords}
% $3$-designs, invariant subspaces, $\mathbb F_q[X,Y]_k$, Cayley transform, Lucas's theorem, Melas code, projective Reed--Solomon codes.
% \end{keywords}

\begin{abstract}
We present a uniform framework for constructing \(3\)-designs from
\(\mathrm{GL}_2(\mathbb F_q)\)-invariant subspaces of \(\mathbb F_q[X,Y]_k\), the space of homogeneous polynomials of degree \(k\). 
Given such a subspace \(W\), we associate a
\(\mathrm{PGL}_2(\mathbb F_q)\)-invariant family of \(k\)-subsets of
\(\mathbb P^1(\mathbb F_q)\). Whenever this family is nonempty, it forms a
\(3\text{-}(q+1,k,\lambda)\) design. 
Via the Cayley transform, the construction is reformulated on the unit circle
\(U_{q+1}\subseteq \mathbb F_{q^2}^{\times}\), where the block conditions
become explicit linear relations among elementary symmetric polynomials.
This reformulation unifies several previously disparate constructions and simplifies a number of delicate ad hoc computations.
When \(k\le q\), the evaluation map on
\(\mathbb P^1(\mathbb F_q)\) identifies \(W\) with a subcode \(C_W\) of the
projective Reed--Solomon code.     
We show that the associated block family is nonempty if and only if \(d(C_W)=q+1-k\). 
Under this condition,  the supports of
minimum-weight codewords in \(C_W\), as well as the supports of suitable
fixed-weight codewords in the dual code \(C_W^\perp\), yield further
\(3\)-designs.

Applying this framework to the Lucas subspaces, which form a distinguished family of invariant subspaces, we obtain explicit block
descriptions, classify the cases in which the defining conditions reduce to a
single equation, and establish several emptiness and nonemptiness results. 
In particular, for \(q=p^e\) and \(k=p^m+1\), we show that the associated block
family is nonempty if and only if \(m\mid e\), in which case it yields the
Steiner system \(S(3,p^m+1,q+1)\). Finally, in the ternary case \(p=3\) and
\(k=7\), we use the weight distribution of the ternary Melas code to
determine the design parameters left
undetermined by Xu et al.
\medskip
\end{abstract}
%\noindent\textbf{Keywords:} $3$-designs, $\mathrm{GL}_2(\mathbb{F}_q)$-invariant subspaces, Lucas subspaces, Cayley transform, homogeneous polynomial spaces

%\begin{keywords}

\section{Introduction}

Let $X$ be a finite set of cardinality $|X|=v$, whose elements are called \emph{points}, and let $\mathcal{B}$ be a collection of $k$-subsets of $X$, whose elements are called \emph{blocks}. 
The pair $(X,\mathcal{B})$ is called a \emph{$t$-$(v,k,\lambda)$ design} if every $t$-subset of $X$ is contained in exactly $\lambda$ blocks of $\mathcal{B}$. The integer $t$ is called the \emph{strength} of the design, and $\lambda$ is called its \emph{index}.
If no block is repeated, then the design is said to be \emph{simple}. 
In this paper, we consider only simple designs. 
Let $b=|\mathcal{B}|$ denote the number of blocks.  Then the parameters satisfy the well-known relation
\begin{equation}\label{solve_b}
    b\binom{k}{t}=\lambda\binom{v}{t}.
\end{equation}
In the special case $\lambda=1$, the design is called a \emph{Steiner system} and is denoted by $S(t,k,v)$.

Let $\mathbb{F}_q$ be a finite field, where $q$ is a prime power.
An $[n,\kappa,d]$ \emph{linear code} $C$ over $\mathbb{F}_q$ is a $\kappa$-dimensional subspace of $\mathbb{F}_q^n$ with minimum Hamming distance $d$. 
The \emph{dual code} of $C$ is defined by 
$C^\perp=\{x\in\mathbb{F}_q^n:x^Tc=0,\forall c\in C\}$. 
For each $0\le i\le n$, let $A_i$ denote the number of codewords of weight $i$ in $C$. Then the sequence $(A_0,A_1,\dots,A_n)$ is called the \emph{weight distribution} of $C$.

The construction of $t$-designs has long drawn on ideas from both permutation group theory and coding theory~\cite{Assmus_Key_1992,Cameron_Lint_1991,Ding2022book,Conway1999}.
One classical line is the \emph{group-theoretic approach}. 
A well-known principle in design theory is that highly transitive permutation groups often give rise to $t$-designs through invariant families of subsets.
This viewpoint has generated many important constructions in design theory, and is closely related to other group-based methods arising from finite geometry and difference sets. 
For $3$-designs, Cameron~\emph{et al.} systematically studied the designs arising from the action of $\PGL(2,q)$ on $k$-subsets of $\mathbb{P}^1(\mathbb{F}_q)$ \cite{Cameron06}, and the corresponding $\mathrm{PSL}(2,q)$ case was developed in \cite{Cameron06-2}. 
This line was further refined for special congruence classes of $q$ and particular block sizes; see, for example, \cite{Balachandran07,Liu12}. 
More recently, Tricot revisited the $\PGL(2,q)$ setting from the viewpoint of explicit multiplicative-type blocks and obtained further families of $3$-designs \cite{Tricot25}. 

Another classical line is the \emph{coding-theoretic approach}. 
Let $X=\{1,2,\dots,n\}$ be the set of coordinate positions. 
Given a linear code $C$ of length $n$, one considers the family $\mathcal{B}_k(C)$ consisting of the supports of all codewords of some fixed Hamming weight $k$. 
Under suitable conditions, the incidence structure $(X,\mathcal{B}_k(C))$ forms a $t$-design. 
A classical starting point of this approach is the Assmus--Mattson theorem \cite{AssmusMattson1969}, together with its later developments and variants, which give sufficient conditions for the supports of codewords of fixed weights in a linear code to form $t$-designs. 
Using the Assmus--Mattson theorem, Ding \cite{Ding18FiveWeight} and Ding and Li \cite{DingLi17} constructed infinite families of $2$-designs and $3$-designs from linear codes. 
% Other coding-theoretic constructions have also been investigated in the literature, including examples from BCH codes and affine-invariant codes \cite{Xiang22,Liu22}.
There are also constructions of t-designs from quadratic functions, APN functions, and other special polynomials \cite{Ding20-2,Xiang20,Tang20,Ling23}.  
Particularly influential in this direction are the works of Ding and Tang, who settled a $70$-year-old problem by constructing an infinite family of near MDS codes over $\mathbb{F}_{3^m}$ supporting $3$-designs and another infinite family over $\mathbb{F}_{2^{2m}}$ supporting $2$-designs~\cite{Ding20}.
Subsequently, Tang and Ding constructed an infinite family of linear codes over $\mathbb{F}_{2^{2m+1}}$ supporting $4$-designs, thereby settling another long-standing problem \cite{Tang21}.
Most recently, Xu~\emph{et al.} constructed several infinite families of $3$-designs from special symmetric polynomials over $\mathbb{F}_{3^m}$ \cite{Xu2024}. 
Nevertheless, many constructions in this direction remain essentially ad hoc, relying on delicate calculations tailored to particular code families or polynomial identities. 
This becomes especially complex when one seeks explicit formulas for the design parameters. 
In particular, in \cite{Xu2024}, the existence of the relevant $3$-designs was established, but the associated parameters $\lambda_1$ and $\lambda_2$ were left unspecified.

Existing constructions from symmetric polynomials and from code supports have largely been developed along separate lines. Our goal is to place these constructions into a common invariant-subspace framework, in which the links among symmetric polynomials, designs, and codes become more transparent.The main contributions of this paper can be summarized as follows.
\begin{itemize}
\item We develop a systematic and unified framework for constructing $3$-designs from $\GL_2(\mathbb F_q)$-invariant subspaces of $\mathbb F_q[X,Y]_k$. 
For each such subspace $W$, we associate a $\PGL_2(\mathbb F_q)$-invariant family $\mathcal B_W$ of $k$-subsets of $\mathbb P^1(\mathbb F_q)$, and prove that every nonempty $\mathcal B_W$ forms a $3$-$(q+1,k,\lambda)$ design. 
Using the Cayley transform, we further reformulate the construction on the unit circle and express the block conditions as explicit linear relations among elementary symmetric polynomials. 
This reformulation clarifies the roles of both symmetric polynomials and supporting codes in the construction, and turns several delicate ad hoc computations into more direct arguments.
% This places several previously disparate constructions into a common conceptual setting, including those arising from symmetric polynomials and those arising as support designs of linear codes.

\item We show that this framework admits an intrinsic coding-theoretic interpretation. 
When $k\le q$, the evaluation map identifies $W$ with a subcode $C_W$ of the projective Reed--Solomon code, and we prove that $\mathcal B_W\neq\varnothing$ if and only if $d(C_W)=q+1-k$. 
In this case, the supports of the minimum-weight codewords of $C_W$ form a $3$-design. 
More generally, for each $w$ with $3\le w\le q+1$, if $C_W^\perp$ contains codewords of weight $w$, then the supports of those codewords also form a $3$-design.

\item We apply the framework to the Lucas subspaces $W_k^{\mathrm{Luc}}$, a special but fairly broad family of invariant subspaces. 
In this setting, we obtain explicit descriptions of the associated block sets, identify when the defining conditions reduce to a single linear equation, and prove several emptiness and nonemptiness results. 
Within this family, we recover, as special cases, the constructions in Theorem~3 of \cite{Tang21} and Theorem~2 of \cite{Xu2024}.
In particular, for $q=p^e$ and $k=p^m+1$, we show that the associated block set is nonempty if and only if $m\mid e$, in which case it yields the Steiner system $S(3,p^m+1,q+1)$. 
Moreover, in the ternary case $p=3$ and $k=7$, the invariant-subspace perspective leads naturally to the ternary Melas code, which enables us to determine the parameters left undetermined in \cite{Xu2024}.
% \item We use the invariant-subspace framework to determine parameters that were left open in earlier constructions. 
% More precisely, by studying the Lucas subspaces $W_k^{\mathrm{Luc}}$, we obtain explicit descriptions of the associated block sets, identify when the defining conditions reduce to a single linear equation, and prove several emptiness and nonemptiness results. 
% In particular, for $q=p^e$ and $k=p^m+1$, we show that the associated block set is nonempty if and only if $m\mid e$, in which case it yields the Steiner system $S(3,p^m+1,q+1)$. 
% Within this class, we recover, as special cases, the constructions in Theorem~3 of \cite{Tang21} and Theorem~2 of \cite{Xu2024}. 
% In particular, in the ternary case $p=3$ and $k=7$, the invariant-subspace perspective leads naturally to the ternary Melas code, which allows us to determine the parameters left undetermined in \cite{Xu2024}.
\end{itemize}

The remainder of this paper is organized as follows.
Section~\ref{section:Preliminaries} recalls the necessary background and introduces the Lucas subspaces.
Section~\ref{section:Construction} develops the general framework for constructing $3$-designs, and gives an alternative description of the construction on the unit circle.
Section~\ref{section:Lucas} applies this framework to Lucas subspaces and studies the associated block sets.
Section~\ref{section:Lucas2} determines the parameters of the designs associated with $W_7^{\mathrm{Luc}}$ via the ternary Melas code, thereby solving the open parameter problem in \cite{Xu2024}.
Finally, Section~\ref{section:Conclusion} concludes the paper.

\section{Preliminaries}\label{section:Preliminaries}
\subsection{Homogeneous polynomial spaces and projective Reed--Solomon codes}
Let $q=p^m$, where $p$ is a prime and $m$ is a positive integer. We denote by $\mathbb{F}_q$ the finite field of $q$ elements, and by $\mathbb{F}_q^\times$ its multiplicative group. 
For an integer $k \ge 0$, let $\mathbb F_q[X,Y]_k$ denote the space of homogeneous polynomials of degree $k$ in the variables $X$ and $Y$ over $\mathbb F_q$, namely
\[
\mathbb F_q[X,Y]_k= \left\{ \sum_{i=0}^k a_i X^{k-i}Y^i : a_i \in \mathbb F_q \right\},
\]
and $\dim_{\mathbb F_q}(\mathbb F_q[X,Y]_k)=k+1$.

We identify the projective line $\mathbb{P}^1(\mathbb{F}_q)$ with the set $\mathbb{F}_q \cup \{\infty\}$, where $x \in \mathbb{F}_q$ corresponds to the homogeneous coordinates $[x:1]$ and $\infty$ corresponds to $[1:0]$.
For $f(X,Y)\in \mathbb F_q[X,Y]_k$, the evaluation of $f$ at a point of $\mathbb{P}^1(\mathbb{F}_q)$ is given by $f(x,1)$ for $x\in\mathbb{F}_q$, and by $f(1,0)$ at the point $\infty$.
Define the $\mathbb{F}_q$-linear evaluation map
\begin{align*}
    \Ev:\mathbb F_q[X,Y]_k&\longrightarrow \mathbb F_q^{\,q+1}\\
    f&\longmapsto \bigl((f(a,1))_{a\in\mathbb F_q},\,f(1,0)\bigr).
\end{align*}
We now determine its kernel.

\begin{lemma}\label{lem:kerEv}
Let $\theta(X,Y):=X^qY-XY^q.$
Then $\theta(X,Y)$ vanishes at every point of $\mathbb P^1(\mathbb F_q)$. Moreover,
\[
\ker(\Ev)=
\begin{cases}
0, & \text{if } k<q+1,\\[2mm]
\theta(X,Y)\cdot \mathbb F_q[X,Y]_{k-(q+1)}, & \text{if } k\ge q+1,
\end{cases}
\]
where $\theta(X,Y)\cdot \mathbb F_q[X,Y]_{k-(q+1)}=\{\theta(X,Y)f(X,Y):f(X,Y)\in\mathbb F_q[X,Y]_{k-(q+1)}\}$.
\end{lemma}

\begin{proof}
For every $x\in\mathbb{F}_q$, we have $\theta(x,1)=x^q-x=0$, and for the point at infinity $\infty$, we have $\theta(1,0)=0$.
Hence $\theta(X,Y)$ vanishes at every point of $\mathbb P^1(\mathbb F_q)$. 
It follows that, for $k\ge q+1$, the  subspace $\theta(X,Y)\cdot \mathbb F_q[X,Y]_{k-(q+1)}\subseteq \ker(\Ev).$
Conversely, let $f(X,Y)=\sum_{i=0}^k a_iX^{k-i}Y^i\in \ker(\Ev)$. 
Since $f(1,0)=0$, we have $a_0=0$, so $f(X,1)=\sum_{i=1}^k a_iX^{k-i}$ is a polynomial of degree at most $k-1$.
On the other hand, $f(x,1)=0$ for all $x\in\mathbb{F}_q$, which implies that $f(X,1)$ is divisible by $X^q-X$. Thus, $f(X,1)=(X^q-X)g(X)$ for some polynomial $g(X)$. 

If $k<q+1$, then $\deg f(X,1)\le k-1<q=\deg(X^q-X)$, and hence $f(X,1)=0$. Therefore $f=0$, and so $\ker(\Ev)=0$.
If $k\ge q+1$, then $\deg g(X)\le k-q-1.$
Let $h(X,Y)\in\mathbb F_q[X,Y]_{k-(q+1)}$ be the $(k-(q+1))$-homogenization of $g(X)$.
Since $a_0=0$, the $k$-homogenization of $f(X,1)$ recovers $f(X,Y)$ precisely. Therefore, 
\[
f(X,Y)
=(X^qY-XY^q)\,h(X,Y)
=\theta(X,Y)\,h(X,Y).
\]
This shows that $f\in \theta(X,Y)\cdot \mathbb F_q[X,Y]_{k-(q+1)}$, completing the proof.
\end{proof}

It follows from Lemma~\ref{lem:kerEv} that, whenever $k\le q$, the evaluation map $\Ev$ is injective. 
In this case, $\Ev$ identifies $\mathbb F_q[X,Y]_k$ with its image in $\mathbb F_q^{\,q+1}$ as an $\mathbb{F}_q$-vector space.
The image $\mathcal P_k:=\im(\Ev)$ is a $[q+1,k+1]$ linear code over $\mathbb F_q$, called the \emph{projective Reed--Solomon code}~\cite{ZhangWanKaipa2019,Lavauzelle2019}.

\begin{proposition}\label{prop:Pk-MDS}
For $k \le q$, the projective Reed--Solomon code $\mathcal{P}_k$ is a $[q+1,k+1,q-k+1]$ MDS code.
\end{proposition}

\begin{proof}
This is standard for projective Reed--Solomon codes; see, for example, \cite{Sorensen1991,Lavauzelle2019}.
\end{proof}

Assume that $k\le q$. Let $W\subseteq \mathbb F_q[X,Y]_k$ be any nonzero $\mathbb{F}_q$-subspace, and let $C_W:=\Ev(W)\subseteq \mathcal{P}_k$ be the corresponding subcode. Then we have the following result.
\begin{corollary}\label{rem:distance-bound}
The minimum distance of $C_W$ satisfies $d(C_W)\ge q+1-k.$ 
\end{corollary}

When $k\le q$, the map $\Ev:\mathbb F_q[X,Y]_k\to \mathcal P_k$ is an isomorphism of $\mathbb F_q$-vector spaces. 
Hence any group action on $\mathbb F_q[X,Y]_k$ transfers naturally to an action on $\mathcal P_k$ via $\Ev$. 

\subsection{Group actions, representations, and invariant subspaces}
In this  subsection, we recall the notions of group actions, representations, and invariant subspaces.
\begin{definition}[Group action~\cite{Artin2011}]
    Let $G$ be a group with identity $e$, and let $X$ be a set. An \emph{action} of $G$ on  $X$ is a map
$G\times X\longrightarrow X,\;
(g,x)\longmapsto g\cdot x,$
such that 
\begin{enumerate}[label=(\roman*)]
    \item $e\cdot x=x$ for all $x\in X$.
    \item $(gg')\cdot x=g\cdot(g'\cdot x)$, for all $g,g'\in G$ and all $x\in X$.
\end{enumerate}
\end{definition}
Equivalently, an action of $G$ on $X$ determines a homomorphism $G\to \Sym(X)$. Thus, by identifying $G$ with its image, we may regard $G$ as a permutation group on $X$.
If $G$ acts on $X$, then it induces a natural action on the power set of $X$.  For any subset $S\subseteq X$, we define $g(S):=\{\,g\cdot x:\ x\in S\,\}$. Similarly, if $\mathcal{S}$ is a family of subsets of $X$, define $g(\mathcal{S}):=\{\,g(S):\ S\in \mathcal{S}\,\}$.
A subset $\mathcal B\subseteq \binom{X}{k}$ is said to be \emph{$G$-invariant} if $g(\mathcal B)=\mathcal B$ for all $g\in G$.
We say that $G$ is \emph{$t$-transitive} on $X$ if it acts transitively on the set of ordered $t$-tuples of distinct elements of $X$. 
It is said to be \emph{sharply $t$-transitive} if, for any two ordered $t$-tuples of distinct elements of $X$, there exists a unique element of $G$ mapping the first tuple to the second. Moreover, $G$ is said to be \emph{$t$-homogeneous} on $X$ if it acts transitively on the set of all $t$-subsets of $X$. 
Clearly, a sharply $t$-transitive group is also $t$-transitive, and a $t$-transitive group is always $t$-homogeneous.

If $G$ acts linearly on  a vector space $V$, we call this action a  representation of $G$.
\begin{definition}[Representation~\cite{Artin2011}]
    Let $\mathbb{F}$ be a field and $V$ be an $\mathbb{F}$-vector space. Let $\GL(V)$ denote the general linear group of $V$. A \emph{representation} of $G$ over $\mathbb{F}$ on $V$ is a group homomorphism $\rho:G\to \GL(V)$. When the homomorphism $\rho$ is clear from the context, we often write $gv=\rho(g)(v)$ for $g\in G,\ v\in V$, and simply refer to $V$ itself as a representation of $G$.
\end{definition}
Let $(\rho, V)$ and $(\rho', W)$ be two representations of $G$. They are said to be \emph{isomorphic} if there exists an invertible linear map $T:V\to W$ such that 
\[
T\circ \rho(g)=\rho'(g)\circ T
\qquad\text{for all }g\in G.
\]
A subspace $U\subseteq V$ is called a \emph{$G$-invariant subspace} (or \emph{subrepresentation}) if
$g(U)\subseteq U$ for all $g\in G$.

We now describe the three group actions used in this paper.\\
(i) \textbf{The action of $\PGL_2(\mathbb F_q)$ on $\mathbb P^1(\mathbb F_q)$.} Let $\GL_2(\mathbb F_q)$ be the group of $2\times 2$ invertible matrices over $\mathbb{F}_q$,  and $\PGL_2(\mathbb F_q)=\GL_2(\mathbb F_q)/\{\lambda I_2:\lambda\in\mathbb F_q^\times\}$. 
    For $g \in \PGL_2(\mathbb F_q),$ let $\widetilde g=
\begin{pmatrix}
a & b\\
c & d
\end{pmatrix}$ be a representative of $g$.
Its action on $\mathbb P^1(\mathbb F_q)$ is given by the fractional linear transformation
\begin{align}\label{PGL_2}
g\cdot x=
\begin{cases}
(ax+b)(cx+d)^{-1}, & \text{if } x\in\mathbb F_q,\ cx+d\neq 0,\\
\infty, & \text{if } x\in\mathbb F_q,\ cx+d=0,\\
ac^{-1}, & \text{if } x=\infty,\ c\neq 0,\\
\infty, & \text{if } x=\infty,\ c=0.
\end{cases}
\end{align}
This action is well-defined, since multiplying the matrix by a nonzero scalar does not alter the induced map on $\mathbb{P}^1(\mathbb{F}_q)$.\\
(ii) \textbf{The action of $\GL_2(\mathbb{F}_q)$ on $\mathbb F_q[X,Y]_k$.} 
For any ${g}=
\begin{pmatrix}
a & b\\
c & d
\end{pmatrix}
\in \GL_2(\mathbb F_q)$ and 
$f(X,Y)\in \mathbb F_q[X,Y]_k$,
define
\begin{equation}\label{op_GL_2}
(g\cdot f)(X,Y)=f(dX-bY,-cX+aY).
\end{equation}
Since $f$ is homogeneous of degree $k$, the polynomial $g\cdot f$ again lies in $\mathbb F_q[X,Y]_k$. This defines a representation of $\GL_2(\mathbb F_q)$ on $\mathbb F_q[X,Y]_k$.\\
(iii) \textbf{The induced action of $\GL_2(\mathbb{F}_q)$ on $\mathcal{P}_k$.}
    Assume that $k\le q$. The action of $\GL_2(\mathbb{F}_q)$ on $\mathbb F_q[X,Y]_k$ induces an action on $\mathcal{P}_k$ via $g\cdot \Ev(f):=\Ev(g\cdot f).$
    Let $g=
    \begin{pmatrix}
    a & b\\
    c & d
    \end{pmatrix}
    \in \GL_2(\mathbb{F}_q).$
    For any finite evaluation point $x\in\mathbb{F}_q$, we have
    \begin{align*}
(g\cdot f)(x,1) &= f(dx-b,-cx+a)\\
&=
\begin{cases}
(a-cx)^k\, f\!\left(\dfrac{dx-b}{a-cx},1\right), & \text{if } a-cx\neq 0,\\[2mm]
(dx-b)^k\, f(1,0), & \text{if } a-cx=0.
\end{cases}
\end{align*}
Since
\[
g^{-1}\cdot x=
\begin{cases}
\dfrac{dx-b}{-cx+a}, & \text{if } -cx+a\neq 0,\\[2mm]
\infty, & \text{if } -cx+a=0,
\end{cases}
\]
this may be rewritten as
\[
(g\cdot f)(x,1)=
\begin{cases}
(a-cx)^k\, f(g^{-1}\cdot x,1), & \text{if } g^{-1}\cdot x\neq \infty,\\[2mm]
(dx-b)^k\, f(1,0), & \text{if } g^{-1}\cdot x=\infty.
\end{cases}
\]
    At the point $x=\infty$, we evaluate at $(1,0)$:
    \begin{align*}
(g\cdot f)(1,0) &= f(d,-c)\\
&=
\begin{cases}
(-c)^k\, f\!\left(-\dfrac{d}{c},1\right), & \text{if } c\neq 0,\\[2mm]
d^k\, f(1,0), & \text{if } c=0.
\end{cases}
\end{align*}
Thus $\GL_2(\mathbb F_q)$ acts on $\mathcal P_k$ by monomial transformations: the coordinates are permuted according to the action on $\mathbb P^1(\mathbb F_q)$, and the scalar factors are induced by the homogeneity of degree $k$. We refer to this as the \emph{monomial action} of $\GL_2(\mathbb F_q)$ on $\mathcal P_k$. Moreover, via the isomorphism $\Ev$, the representations $\mathbb F_q[X,Y]_k$ and $\mathcal P_k$ are isomorphic as $\GL_2(\mathbb F_q)$-representations.
It is worth noting that the case $k=q-1$ is particularly distinguished.
Since $v^{q-1}=1$ for every $v\in\mathbb F_q^\times$, all the scalar factors in the above evaluation formulas become $1$ when $k=q-1$.
More precisely, if $g\in \GL_2(\mathbb F_q)$ and $u=(u_x)_{x\in\mathbb P^1(\mathbb F_q)}\in \mathcal P_{q-1}$, then the action simplifies to
\[
(g\cdot u)_x=u_{g^{-1}\cdot x}
\qquad\text{for all }x\in\mathbb P^1(\mathbb F_q).
\]
Thus, the monomial action on $\mathcal P_{q-1}$ reduces exactly to the permutation action on the coordinate positions and hence factors through $\PGL_2(\mathbb F_q)$.

% Finally, if $W\subseteq \Sym^k(\mathbb F_q^2)$ is a $GL_2(\mathbb F_q)$-invariant subspace, then
% \[
% C_W:=\Ev(W)\subseteq \mathcal P_k
% \]
% is a subcode stable under the induced monomial action of $GL_2(\mathbb F_q)$. This observation will be the starting point of the constructions in the next section. :contentReference[oaicite:4]{index=4}
% %\section{Designs from $GL_2(\mathbb{F}_q)$-Invariant Subspaces and Their Associated Codes}
\subsection{Lucas subspaces: a family of $\GL_2(\mathbb F_q)$-invariant subspaces of $\mathbb F_q[X,Y]_k$}
We conclude this section by introducing the Lucas subspaces, a family of $\GL_2(\mathbb F_q)$-invariant subspaces of $\mathbb F_q[X,Y]_k$ with respect to the action defined in \eqref{op_GL_2}. 
They will serve as the main family of invariant subspaces in Sections~\ref{section:Lucas} and~\ref{section:Lucas2}. 
This subsection may be skipped on a first reading and consulted when the Lucas subspaces first appear later in the paper.

We begin with some notation for base-$p$ expansions.
For each integer $i$ with $0\le i\le k$, write the base-$p$ expansions of $k$ and $i$ in the form
$$k=\sum_{r\ge 0}k_rp^r,\qquad i=\sum_{r\ge 0}i_rp^r \qquad\text{for}\;\; 0\le k_r,i_r\le p-1.  $$
Define a partial order on $\{0,1,\cdots,k\}$ by $i\le_p k$ if $ i_r\le k_r$ for all $r$.

\begin{definition}[Lucas subspace]\label{def:Lucas-subspace}
For an integer $k\ge 0$, define the  Lucas subspace $W_k^{\mathrm{Luc}}$ by
$$
W_k^{\mathrm{Luc}}:=\Span_{\mathbb F_q}\{X^{k-i}Y^i:0\le i\le k,\ i\le_p k\}\subseteq \mathbb F_q[X,Y]_k.
$$
\end{definition}
The dimension of $W_k^{\mathrm{Luc}}$ is $\prod_{r\ge 0}(k_r+1)$, since the monomials $X^{k-i}Y^i$ with $i\le_p k$ form a basis of $W_k^{\mathrm{Luc}}$, and the condition $i\le_p k$  is equivalent to $0\le i_r\le k_r$ for all $r$.
In particular, if $k=ap^m-1$ with $1\le a<p$,  then  $k=(p-1)+(p-1)p+\cdots +(p-1)p^{m-1}+(a-1)p^m.$
Hence $k_r=p-1$ for $r<m$ and $k_m=a-1$. 
It follows that every integer $0\le i\le k$ satisfies $i\le_p k$, and therefore
$W_k^{\mathrm{Luc}}=\mathbb F_q[X,Y]_k.$

We use Lucas's theorem to prove that $W_k^{\mathrm{Luc}}$ is $\GL_2(\mathbb F_q)$-invariant.
\begin{lemma}[Lucas\cite{lucas1878}]\label{lem:Lucas}
Let $m=\sum_{r\ge 0}m_rp^r,\ell=\sum_{r\ge 0}\ell_rp^r,$ where $0\le m_r,\ell_r\le p-1$.
Then $$\binom{m}{\ell}\equiv \prod_{r\ge 0}\binom{m_r}{\ell_r}\pmod p.$$
In particular, $\binom{m}{\ell}\not\equiv 0\pmod p$ if and only if $\ell\le_p m.$
\end{lemma}
We now prove the $\GL_2(\mathbb F_q)$-invariance of the Lucas subspaces.
\begin{theorem}\label{thm:Lucas-invariant}
The subspace $W_k^{\mathrm{Luc}}$ is a $\GL_2(\mathbb F_q)$-invariant subspace of $\mathbb F_q[X,Y]_k$.
\end{theorem}

\begin{proof}
Let $f_i(X,Y)=X^{k-i}Y^i$ with $i\le_p k$. For any $g=\begin{pmatrix}a&b\\ c&d\end{pmatrix}\in \GL_2(\mathbb F_q),$
we have
\begin{align*}
    (g\cdot f_i)(X,Y)&=(dX-bY)^{k-i}(-cX+aY)^i\\
    &=\sum_{u=0}^{k-i}\sum_{v=0}^i \binom{k-i}{u}\binom{i}{v}a^v(-b)^u(-c)^{i-v} d^{k-i-u} X^{k-(u+v)}Y^{u+v}
\end{align*}
It suffices to show that whenever the coefficient of $X^{k-j}Y^{j}$ is nonzero, one has $j\le_p k$.
So assume that the coefficient of $X^{k-j}Y^j$ is nonzero. 
Then there exists a pair $(u,v)$ with $u+v=j$ such that the corresponding summand is nonzero. In particular,
$\binom{k-i}{u}\binom{i}{v}\not\equiv 0\pmod p.$
By Lemma~\ref{lem:Lucas}, this implies $u\le_p k-i$ and $v\le_p i$.
Since $i\le_p k$, the subtraction $k-i$ involves no borrowing in base $p$.
So $u_r\le(k-i)_r= k_r-i_r$ and $v_r\le i_r$ for all $r$.
Hence
$$u_r+v_r\le (k_r-i_r)+i_r=k_r<p$$ for all $r$. 
Thus the addition $j=u+v$ involves no carrying in base $p$, and therefore $j_r=u_r+v_r\le k_r$ for all $r$. 
Thus every monomial occurring in $g\cdot f_i$ is of the form $X^{k-j}Y^j$ with $j\le_p k$. Hence $g\cdot f_i\in W_k^{\mathrm{Luc}}$.
Since the monomials $f_i$ with $i\le_p k$ span $W_k^{\mathrm{Luc}}$, it follows that
$g\cdot W_k^{\mathrm{Luc}}\subseteq W_k^{\mathrm{Luc}}$ for all $g\in \GL_2(\mathbb F_q)$. This proves the theorem.
\end{proof}

The subspace $W_k^{\mathrm{Luc}}$ also admits a natural representation-theoretic interpretation. Let
$E=\mathbb F_q^2$ with standard basis $\{e_1,e_2\}$, and let $\{X,Y\}\subset E^*$ be the dual basis.
Then the space $\mathbb F_q[X,Y]_k$ of homogeneous polynomials of degree $k$ may be identified with
the usual symmetric power $\Sym^k(E^*)$.

Following the terminology of McDowell and Wildon \cite{mcdowell2022modular}, define the \emph{lower symmetric power} $\Sym_k(E^*):=((E^*)^{\otimes k})^{S_k},$
that is, the subspace of invariants under the place permutation action of the symmetric group $S_k$.
Consider the canonical composite
\[
\Sym_k(E^*)\hookrightarrow (E^*)^{\otimes k}\twoheadrightarrow \Sym^k(E^*)
\cong \mathbb F_q[X,Y]_k.
\]

For $0\le a\le k$, let $(X^{\otimes k-a}\otimes Y^{\otimes a})_{\mathrm{sym}}$
denote the sum of all distinct permutations of $X^{\otimes k-a}\otimes Y^{\otimes a}$.
Under the above map, one has
$(X^{\otimes k-a}\otimes Y^{\otimes a})_{\mathrm{sym}}
\longmapsto
\binom{k}{a}X^{k-a}Y^a.$
It follows that the image is $\Span\{X^{k-a}Y^a:\binom{k}{a}\not\equiv 0\pmod p\}.$
By Lemma~\ref{lem:Lucas}, this is precisely $\Span\{X^{k-a}Y^a:a\le_p k\}=W_k^{\mathrm{Luc}}.$
Thus $W_k^{\mathrm{Luc}}$ is not an ad hoc construction, but the image of the natural map from the
lower symmetric power to the usual symmetric power.

\section{A general framework for constructing 3-designs from $\GL_2(\mathbb F_q)$-invariant subspaces}\label{section:Construction}
In this section, we develop a general framework for constructing $3$-designs from $\GL_2(\mathbb F_q)$-invariant subspaces. We first define the associated block families and show that each nonempty block family gives rise to a $3$-design. We then relate these constructions to associated codes and their duals, and finally reformulate the block conditions in unit-circle coordinates via the Cayley transform.

\subsection{3-designs from invariant subspaces}
We begin with two basic lemmas.
\begin{lemma}({\cite[Theorem 1.1]{BethJungnickelLenz1999}})\label{lemma:design}
Let $G$ be a permutation group on a finite set $X$ with $|X|=v$. Suppose that $G$ is $t$-homogeneous on $X$, and let $\mathcal B$ be a nonempty $G$-invariant subset of $\binom{X}{k}$. Then $(X,\mathcal B)$ is a $t$-$(v,k,\lambda)$ design for some $\lambda$. Moreover, $G$ acts as an automorphism group of this design.
\end{lemma}

\begin{lemma}[{\cite[Corollary 2.5]{cameron2000notes}}]\label{lemma:3-transitive}
The group $\PGL_2(\mathbb F_q)$, acting on $\mathbb P^1(\mathbb F_q)$ as in \eqref{PGL_2}, is sharply $3$-transitive.
In particular, it acts $3$-homogeneously on $\mathbb P^1(\mathbb F_q)$.
\end{lemma}
Let $X=\mathbb P^1(\mathbb F_q)=\mathbb F_q\cup\{\infty\},$ $G=\PGL_2(\mathbb F_q),$ and $ V=\mathbb F_q[X,Y]_k$.
Let $W$ be a $\GL_2(\mathbb F_q)$-invariant subspace of $V$. For each $k$-subset $S\subseteq X$, define the polynomial
$$
F_S(X,Y):=\prod_{t\in S}(X-tY)\in V,
$$
where the linear factor corresponding to $t=\infty$ is taken to be $Y$.
Define the associated block set
$$
\mathcal B_W:=\left\{\,S\in \binom{X}{k}: F_S\in W\,\right\}.
$$
In what follows, we always assume that $k\le q+1$.
We now show that, whenever $\mathcal B_W$ is nonempty, the incidence structure $(X,\mathcal B_W)$ is a $3$-design.
\begin{proposition}\label{prop:BW-design}
If $\mathcal B_W\neq\varnothing$, then $(X,\mathcal B_W)$ is a $3$-$(q+1,k,\lambda)$ design for some $\lambda$.
\end{proposition}

\begin{proof}
By Lemmas~\ref{lemma:design} and \ref{lemma:3-transitive}, it suffices to show that $\mathcal B_W$ is $G$-invariant. For any $S\in\mathcal B_W$ and $g\in \PGL_2(\mathbb F_q)$, choose a representative $\widetilde g=\begin{pmatrix} a&b\\ c&d\end{pmatrix}\in \GL_2(\mathbb F_q)$.
If $S\subseteq \mathbb F_q$, then
\begin{align*}
    \widetilde g\cdot F_S&=\prod_{t\in S}((dX-bY)-t(-cX+aY))\\
    &=\prod_{t\in S}((ct+d)X-(at+b)Y)\\
    &=\mu F_{g(S)},
\end{align*}
for some $\mu\in\mathbb F_q^*$. If $\infty\in S$, then $F_S=Y\prod_{t\in S\setminus\{\infty\}}(X-tY),$ and one checks similarly that $\widetilde g\cdot F_S=\mu F_{g(S)}$ for some $\mu\in\mathbb F_q^*$. Since $\widetilde g\cdot F_S\in W$ and $W$ is an $\mathbb F_q$-subspace, we conclude that $F_{g(S)}\in W$. Hence $g(S)\in \mathcal B_W$, and therefore $\mathcal B_W$ is $G$-invariant.
\end{proof}

The next proposition gives several equivalent characterizations of the nonemptiness of $\mathcal B_W$.
\begin{proposition}\label{prop:BW-nonempty}
Assume that $k\le q+1$. Then the following are equivalent:
\begin{enumerate}[label=(\roman*)]
    \item $\mathcal B_W\neq\varnothing$;
    \item there exists a $k$-subset $S\subseteq X$ such that $F_S\in W$;
    \item $W$ contains a nonzero polynomial $f$ vanishing at exactly $k$ distinct points of $X$.
\end{enumerate}
If, in addition, $k\le q$, then the above conditions are also equivalent to
\begin{enumerate}[label=(\roman*),start=4]
    \item $d(C_W)=q+1-k$.
\end{enumerate}
%Moreover, whenever these conditions hold, $(X,\mathcal B_W)$ is a $3$-$(q+1,k,\lambda)$ design for some $\lambda$.
\end{proposition}

\begin{proof}
The equivalence of $(i)$ and $(ii)$ is immediate from the definition of $\mathcal B_W$.
If $(ii)$ holds, then $F_S\in W$ for some $S\in\binom{X}{k}$, and by construction, $F_S$ vanishes exactly at the $k$ distinct points of $S$. 
Hence $(ii)\Rightarrow(iii)$.
Conversely, assume $(iii)$, and let $f\in W$ be a nonzero homogeneous polynomial of degree $k$ vanishing at exactly $k$ distinct points of $X$, say $S\subseteq X$. 
Since $f$ has degree $k$ and all its zeros in $X$ are simple, it follows that $f=\mu F_S$ for some $\mu\in\mathbb F_q^\times$. 
Therefore $F_S\in W$, and hence $(iii)\Rightarrow(ii)$. 
This proves the equivalence of $(i)$, $(ii)$ and $(iii)$.

Now assume  that $k\le q$. 
We show that $(iii)$ and $(iv)$ are equivalent. 
Let $0\neq f\in W$. Then the zero coordinates of $\Ev(f)\in C_W$ are precisely the points of $X$ at which $f$ vanishes. 
Hence
$$
\wt(\Ev(f))=(q+1)-\#\{x\in X:f(x)=0\}.
$$
Therefore, $f$ vanishes at exactly $k$ distinct points of $X$ if and only if $\Ev(f)$ has weight $q+1-k$. 
By Corollary~\ref{rem:distance-bound}, every nonzero codeword of $C_W$ has weight at least $q+1-k$. 
Consequently, $C_W$ contains a codeword of weight $q+1-k$ if and only if $d(C_W)=q+1-k.$
Thus $(iii)$ and $(iv)$ are equivalent, completing the proof.
\end{proof}

\subsection{More 3-designs from related codes}
We next show that the same invariant-subspace construction also gives rise to $3$-designs from related codes. More precisely, we consider the support designs arising from minimum-weight codewords of $C_W$ and fixed-weight codewords of $C_W^\perp$.

For a codeword $c=(c_x)_{x\in X}$, define $\Supp(c):=\{\,x\in X:c_x\neq 0\,\}.$ We have the following lemma.
\begin{lemma}\label{cor:minwt-design}
Assume that $k\le q$. If $\mathcal B_W\neq\varnothing$, equivalently if $d(C_W)=q+1-k$, then the supports of the minimum-weight codewords in $C_W$ form a $3$-$(q+1,q+1-k,\lambda)$ design for some $\lambda$.
\end{lemma}

\begin{proof}
By Proposition~\ref{prop:BW-nonempty}, the hypothesis implies that $(X,\mathcal B_W)$ is a $3$-$(q+1,k,\mu)$ design for some $\mu$. For each minimum-weight codeword $c=\Ev(f)\in C_W$, its zero set $Z(c):=\{\,x\in X:c_x=0\,\}$ is a block of $\mathcal B_W$. Since $ \Supp(c)=X\setminus Z(c)$, the supports of the minimum-weight codewords are precisely the complements of the blocks in $\mathcal B_W$. Since the complements of the blocks in a $t$-design again form a $t$-design, the conclusion follows.
\end{proof}

% \begin{remark}
% Later we shall see that Corollary~\ref{cor:minwt-design} recovers Theorem~40 of \cite{tang2020infinite} and Theorems~4--5 of \cite{xu2024infinite} as special cases.
% \end{remark}

We now turn to the dual code $C_W^\perp$.
\begin{proposition}\label{prop:dual-design}
Assume that $k\le q$. Then the following hold:
\begin{enumerate}[label=(\roman*)]
    \item $d(C_W^\perp)\le k+2$.
    \item For every integer $w$ with $3\le w\le q+1$, if $C_W^\perp$ contains codewords of weight $w$, then the supports of all codewords of weight $w$ in $C_W^\perp$ form a $3$-$(q+1,w,\lambda_w)$ design for some $\lambda_w$.
\end{enumerate}
In particular, if $d(C_W^\perp)\ge 3$, then the supports of the minimum-weight codewords in $C_W^\perp$ form a $3$-$(q+1,d(C_W^\perp),\lambda)$ design for some $\lambda$.
\end{proposition}
\begin{proof}
Since $C_W\subseteq \mathcal P_k$, taking duals yields $\mathcal P_k^\perp\subseteq C_W^\perp$. By Proposition~\ref{prop:Pk-MDS}, the code $\mathcal P_k$ is MDS with parameters $[q+1,k+1,q+1-k]$, and hence its dual $\mathcal P_k^\perp$ is also MDS, with parameters $[q+1,q-k,k+2]$. In particular, $d(\mathcal P_k^\perp)=k+2$. Since $\mathcal P_k^\perp\subseteq C_W^\perp$, we obtain $d(C_W^\perp)\le d(\mathcal P_k^\perp)=k+2$, proving $(i)$.

For a fixed integer $w$, let
$\mathcal B_w(C_W^\perp):=\{\Supp(c):c\in C_W^\perp,\ \wt(c)=w\}.$
To prove $(ii)$, it suffices to show that $\mathcal B_w(C_W^\perp)$ is $\PGL_2(\mathbb F_q)$-invariant. Let $S=\Supp(y)\in \mathcal B_w(C_W^\perp), $where $y\in C_W^\perp$ and $\wt(y)=w$. Take any $g\in \PGL_2(\mathbb F_q)$, and choose a representative
$\widetilde g\in \GL_2(\mathbb F_q)$ of $g$. Since $\GL_2(\mathbb F_q)$ acts on $\mathcal P_k$ by monomial transformations, there exists a monomial matrix
$M=DP$ such that
$\widetilde g\cdot c=Mc$ for all $c\in \mathcal P_k,$
where $D$ is diagonal with nonzero diagonal entries and $P$ is a permutation matrix. 
Now define
$z:=M^{-T}y.$
We claim that $z\in C_W^\perp$. Indeed, for any $c\in C_W$, the $\GL_2(\mathbb{F}_q)$-invariance of $C_W$ implies that $M^{-1}c\in C_W$. Thus, we have $$c^Tz = c^TM^{-T}y = (M^{-1}c)^Ty = 0,$$ since $y\in C_W^\perp$. This confirms that $z\in C_W^\perp$.
Since
$M^{-T}=D^{-1}P,$
the matrix $M^{-T}$ is again monomial. Thus it preserves Hamming weight, and
$\wt(z)=\wt(y)=w.$
Moreover, multiplication by $D^{-1}$ does not change the support, while the permutation matrix $P$ acts on coordinates as $g$. Therefore
$$\Supp(z)=g(\Supp(y))=g(S).$$
Hence
$g(S)\in \mathcal B_w(C_W^\perp).$
This shows that $\mathcal B_w(C_W^\perp)$ is $G$-invariant. If $C_W^\perp$ contains a codeword of weight $w$, then
$\mathcal B_w(C_W^\perp)\neq\varnothing$. Therefore, by Lemmas~\ref{lemma:design} and \ref{lemma:3-transitive}, $(X,\mathcal B_w(C_W^\perp))$ is a $3$-$(q+1,w,\lambda_w)$ design for some $\lambda_w$. This proves $(ii)$.
The final assertion follows by taking
$w=d(C_W^\perp),$
provided that
$d(C_W^\perp)\ge 3.$
\end{proof}

% \begin{remark}
% Later we shall see that Proposition~\ref{prop:dual-design} recovers Theorem~40 of \cite{tang2020infinite} and Theorems~3--4 of \cite{xu2024infinite} as special cases.
% \end{remark}

\subsection{An alternative description in unit-circle coordinates}\label{section:Reformulation}
Let $U_{q+1}:=\{u\in\mathbb F_{q^2}^\times:u^{q+1}=1\}$ be the unit circle in $\mathbb F_{q^2}$. 
We now give an equivalent reformulation of the block set $\mathcal B_W$ in unit-circle coordinates via the Cayley transform. 
This will translate the geometric block conditions into explicit linear conditions on elementary symmetric polynomials.

Choose $\xi\in\mathbb F_{q^2}\setminus\mathbb F_q$, and define the Cayley transform
$$
\kappa:\mathbb P^1(\mathbb F_q)\longrightarrow U_{q+1},\qquad
\kappa(x)=\frac{x-\xi}{x-\xi^q}\ \ (x\in\mathbb F_q),\qquad \kappa(\infty)=1.
$$
For $x\in\mathbb F_q$, we have $\kappa(x)^q=\frac{x-\xi^q}{x-\xi}=\kappa(x)^{-1},$ so $\kappa(x)\in U_{q+1}$. 
Since $\kappa$ is induced by an invertible linear fractional transformation, it is injective on $\mathbb P^1(\mathbb F_{q^2})$. As $\kappa(\mathbb P^1(\mathbb F_q))\subseteq U_{q+1}$ and both sets have cardinality $q+1$, it follows that $\kappa$ is a bijection from $\mathbb P^1(\mathbb F_q)$ onto $U_{q+1}$.
In homogeneous coordinates, $\kappa$ and its inverse are given by
$$[X:Y]\longmapsto [U:V]=[X-\xi Y:X-\xi^qY];\qquad
[U:V]\longmapsto [X:Y]=[\xi^qU-\xi V:U-V].$$
One may choose $\xi$ as follows. If $p>2$, choose $\xi$ such that $\xi^q=-\xi$. 
Then we have $\kappa(x)=\frac{x-\xi}{x+\xi}.$
If $p=2$, choose $\xi$ such that $\xi^q+\xi=1$. Then $\kappa(x)=\frac{x+\xi}{x+\xi^q}.$

Since the Cayley transform is defined over $\mathbb F_{q^2}$, we extend scalars from $\mathbb F_q$ to $\mathbb F_{q^2}$.
Recall that $V=\mathbb F_q[X,Y]_k$. 
Set $ V_{\mathbb F_{q^2}}:=V\otimes_{\mathbb F_q}\mathbb F_{q^2}.$
Then there is a natural $\mathbb F_{q^2}$-linear isomorphism
$$
V_{\mathbb F_{q^2}}{\longrightarrow}\mathbb F_{q^2}[X,Y]_k,
\qquad
\Big(\sum_{i=0}^k a_iX^{k-i}Y^i\Big)\otimes\lambda\longmapsto\sum_{i=0}^k (\lambda a_i)X^{k-i}Y^i.
$$
Via this identification, we regard $V_{\mathbb F_{q^2}}$ as $\mathbb F_{q^2}[X,Y]_k$. Similarly, for any $\mathbb F_q$-subspace $W\subseteq V$, we define
$$
W_{\mathbb F_{q^2}}:=W\otimes_{\mathbb F_q}\mathbb F_{q^2}\subseteq \mathbb F_{q^2}[X,Y]_k.
$$

Let
$
H=\begin{pmatrix}
-1 & \xi\\
-1 & \xi^q
\end{pmatrix}\in \GL_2(\mathbb F_{q^2}).
$
Then for every $f\in \mathbb F_{q^2}[X,Y]_k$,
$(H\cdot f)(U,V)=f(\xi^qU-\xi V,U-V).$
We define the transformed subspace associated with $W$ by
$$
\widetilde W:=H\cdot W_{\mathbb F_{q^2}}\subseteq \mathbb F_{q^2}[U,V]_k.
$$
Equivalently,
$
\widetilde W=\Span_{\mathbb F_{q^2}}\{\,f(\xi^qU-\xi V,U-V):f\in W\,\}.
$
Thus $\widetilde W$ is precisely the image of the scalar extension $W_{\mathbb F_{q^2}}$ under the $\mathbb F_{q^2}$-linear automorphism induced by $H$.

For a $k$-subset $T\subseteq U_{q+1}$, define $ G_T(U,V):=\prod_{u\in T}(U-uV)\in \mathbb F_{q^2}[U,V]_k,$
and set
$\widetilde{\mathcal B}_W:=\{\kappa(S):S\in\mathcal B_W\}\subseteq \binom{U_{q+1}}{k}.
$ We have the following proposition.

\begin{proposition}\label{prop:Cayley-block}
Let $W$ be a $\GL_2(\mathbb F_q)$-invariant subspace of $V=\mathbb F_q[X,Y]_k$, and assume that $k\le q+1$. Then
$$
\widetilde{\mathcal B}_W=\left\{\,T\in\binom{U_{q+1}}{k}:G_T(U,V)\in\widetilde W\,\right\}.
$$
In particular, if $\mathcal B_W\neq\varnothing$, then $(U_{q+1},\widetilde{\mathcal B}_W)$ is a $3$-$(q+1,k,\lambda)$ design for some $\lambda$.
\end{proposition}

\begin{proof}
For $t\in\mathbb F_q$, one has
$$
X-tY=(\xi^qU-\xi V)-t(U-V)=(\xi^q-t)U-(\xi-t)V=(\xi^q-t)(U-\kappa(t)V).
$$
For $t=\infty$, the corresponding factor is $Y=U-V=U-\kappa(\infty)V.$
Therefore, for $S\subseteq \mathbb P^1(\mathbb F_q)$ and $|S|=k$, we have 
$$
F_S(\xi^qU-\xi V,U-V)=c_S\,G_{\kappa(S)}(U,V)
$$
for some nonzero scalar $c_S\in\mathbb F_{q^2}^\times$.
Then, $F_S\in W$ if and only if $c_S\,G_{\kappa(S)}(U,V)\in \widetilde W$. Since $\widetilde W$ is an $\mathbb F_{q^2}$-vector subspace, it is equivalent to $G_{\kappa(S)}(U,V)\in \widetilde W$. By definitions of $\mathcal B_W$ and $\widetilde {\mathcal B}_W$, we have $\kappa(S)\in \widetilde {\mathcal B}_W$ if and only if $G_{\kappa(S)}\in \widetilde W$. 
It follows that
$$
\widetilde{\mathcal B}_W
=\{\kappa(S):S\in\mathcal B_W\}
=\left\{\,T\in\binom{U_{q+1}}{k}:G_T(U,V)\in\widetilde W\,\right\}.
$$
Since $\kappa$ is a bijection from $\mathbb P^1(\mathbb F_q)$ onto $U_{q+1}$, the incidence structure
$(U_{q+1},\widetilde{\mathcal B}_W)$
is isomorphic to $(\mathbb P^1(\mathbb F_q),\mathcal B_W).$
Therefore, if $\mathcal B_W\neq\varnothing$, Proposition~\ref{prop:BW-design} implies that $(U_{q+1},\widetilde{\mathcal B}_W)$
is a $3$-$(q+1,k,\lambda)$ design for some $\lambda$.
\end{proof}

The advantage of this reformulation is that the coefficients of $G_T(U,V)$ are elementary symmetric polynomials in the elements of $T$. More precisely, if $T=\{u_1,\dots,u_k\}\subseteq U_{q+1}$, then
\begin{equation}\label{ReG_T}
    G_T(U,V)=\prod_{u\in T}(U-uV)=\sum_{a=0}^k(-1)^ae_a(T)U^{k-a}V^a,
\end{equation}
where $e_a(T)=e_a(u_1,\dots,u_k)$ denotes the $a$-th elementary symmetric polynomial in the elements of $T$.

\begin{proposition}\label{prop:Cayley-coeff}
Let $m$ be a positive integer. Suppose that the transformed subspace $\widetilde W\subseteq \mathbb F_{q^2}[U,V]_k$ is given by a system of $m$ linear conditions on the coefficients, say
$$
\widetilde W=\left\{\sum_{a=0}^k c_aU^{k-a}V^a:\sum_{a=0}^k\lambda_{r,a}c_a=0\ \text{for }1\le r\le m\right\},
$$
where $\lambda_{r,a}\in\mathbb F_{q^2}$. Then
$$
\widetilde{\mathcal B}_W=
\left\{\,T\in\binom{U_{q+1}}{k}:\sum_{a=0}^k(-1)^a\lambda_{r,a}e_a(T)=0\ \text{for }1\le r\le m\right\}.
$$
\end{proposition}

\begin{proof}
By Proposition~\ref{prop:Cayley-block}, a $k$-subset $T\subseteq U_{q+1}$ lies in $\widetilde{\mathcal B}_W$ if and only if $G_T(U,V)\in\widetilde W$. By (\ref{ReG_T}),
this is equivalent to the coefficient conditions
$\sum_{a=0}^k(-1)^a\lambda_{r,a}e_a(T)=0,$ 
for $1\le r\le m$.
This completes the proof.
\end{proof}

\begin{remark}
The reformulation above is most useful when $\widetilde W$ admits a simple coefficient description. In that case, the block set $\widetilde{\mathcal B}_W$ is determined by explicit linear equations in the elementary symmetric polynomials of the points of $T$. Of course, to determine the blocks themselves one must also impose the conditions that the coordinates lie in
$U_{q+1}=\{u\in\mathbb F_{q^2}^\times:u^{q+1}=1\}$
and are pairwise distinct. Thus the $U_{q+1}$-model gives a more symmetric reformulation, while concrete calculations may still be more convenient in the original $\mathbb P^1(\mathbb F_q)$-model.
\end{remark}

\section{The designs arising from Lucas subspaces}\label{section:Lucas}
In this section, we apply the general framework developed in Section~\ref{section:Construction} to the Lucas subspaces. 
We first derive explicit descriptions of the associated block sets, including the cases in which the Cayley description reduces to a single equation.
We then study when these block sets are empty or nonempty, construct several basic families of blocks, and completely determine the case $k=p^m+1$.

\subsection{Explicit descriptions of the block sets}
We now specialize the reformulation of Section~\ref{section:Reformulation} to $W_k^{\mathrm{Luc}}$. 
In this case, the transformed coefficient conditions become particularly simple.
\begin{theorem}\label{thm:Lucas-Cayley}
Let $\kappa$ be the Cayley transform introduced in Section~\ref{section:Reformulation}, and let
$
\widetilde{\mathcal B}_{W_k^{\mathrm{Luc}}}=\{\kappa(S):S\in  \mathcal{B}_{W_k^{\mathrm{Luc}}}\}\subseteq \binom{U_{q+1}}{k}.
$
Then
$$
\widetilde{\mathcal B}_{W_k^{\mathrm{Luc}}}
=
\left\{
T\in \binom{U_{q+1}}{k}: e_a(T)=0\ \text{for all }a\not\le_p k
\right\}.
$$
\end{theorem}

\begin{proof}
Let
$H=
\begin{pmatrix}
-1 & \xi\\
-1 & \xi^q
\end{pmatrix}\in \GL_2(\mathbb F_{q^2}).$
Then $(H\cdot f)(U,V)=f(\xi^qU-\xi V,U-V)$ for every $f\in \mathbb F_{q^2}[X,Y]_k.$
Extend scalars from $\mathbb F_q$ to $\mathbb F_{q^2}$ and set
$$
W_{k,\mathbb F_{q^2}}^{\mathrm{Luc}}
:=
\Span_{\mathbb F_{q^2}}\{U^{k-a}V^a:a\le_p k\}
\subseteq \mathbb F_{q^2}[U,V]_k.
$$
Since the proof of Theorem~\ref{thm:Lucas-invariant} uses only Lemma~\ref{lem:Lucas} and the binomial expansion, it remains valid over any extension field of characteristic $p$. Hence $W_{k,\mathbb F_{q^2}}^{\mathrm{Luc}}$ is invariant under $\GL_2(\mathbb F_{q^2})$, and in particular
$H\cdot W_{k,\mathbb F_{q^2}}^{\mathrm{Luc}}=W_{k,\mathbb F_{q^2}}^{\mathrm{Luc}}.$
Thus, in the notation of Section~\ref{section:Reformulation}, the transformed subspace $\widetilde{W}$ attached to $W_k^{\mathrm{Luc}}$ is precisely $W_{k,\mathbb F_{q^2}}^{\mathrm{Luc}}$.
By Proposition~\ref{prop:Cayley-block}, a $k$-subset $T\subseteq U_{q+1}$ lies in $\widetilde{\mathcal B}_{W_k^{\mathrm{Luc}}}$ if and only if $G_T(U,V)$
belongs to $W_{k,\mathbb F_{q^2}}^{\mathrm{Luc}}$. Writing
$G_T(U,V)=\sum_{a=0}^k(-1)^ae_a(T)U^{k-a}V^a,$ we see that $G_T(U,V)\in W_{k,\mathbb F_{q^2}}^{\mathrm{Luc}}$ if and only if the coefficients of $U^{k-a}V^a$ vanish for all $a\not\le_p k$, that is, if and only if $e_a(T)=0$ for all $a\not\le_p k$.
\end{proof}

The conditions in Theorem~\ref{thm:Lucas-Cayley} admit a useful symmetry. Since the points of $T$ lie on the unit circle, the conditions indexed by $a$ and $k-a$ are equivalent.

\begin{lemma}\label{lem:Lucas-symmetry}
Let $T\in \binom{U_{q+1}}{k}$. Then the conditions $e_a(T)=0$ for $a\not\le_p k$
are equivalent to the smaller system $e_a(T)=0$ for $a\not\le_p k$ and $1\le a\le \lfloor k/2\rfloor.$
\end{lemma}

\begin{proof}
Write $T=\{u_1,\dots,u_k\}\subseteq U_{q+1}$. Since $u_i^q=u_i^{-1}$ for every $i\in\{1,\cdots,k\}$, we have
$e_a(T)^q=e_a(u_1^q,\dots,u_k^q)=e_a(u_1^{-1},\dots,u_k^{-1}),$ for each $0\le a\le k$.
On the other hand, for nonzero $u_1,\dots,u_k$,
$$
e_a(u_1^{-1},\dots,u_k^{-1})=\frac{e_{k-a}(u_1,\dots,u_k)}{e_k(u_1,\dots,u_k)}.
$$
Hence,  $e_{k-a}(T)=e_k(T)e_a(T)^q.$
Since $e_k(T)=\prod_{u\in T}u\neq 0$, it follows that $e_a(T)=0$ is equivalent to $e_{k-a}(T)=0.$
Moreover, by Lemma~\ref{lem:Lucas}, we have 
$a\le_p k$ if and only if $\binom{k}{a}\not\equiv 0 \pmod p$, and  similarly
$k-a\le_p k$ if and only if $\binom{k}{k-a}\not\equiv 0 \pmod p$.
Since $\binom{k}{a}=\binom{k}{k-a}$, it follows that
$a\le_p k$  exactly when $k-a\le_p k,$
and hence $a\not\le_p k$  exactly when  $k-a\not\le_p k.$
Thus the vanishing conditions occur in pairs $(a,k-a)$, and it suffices to impose one condition from each pair, namely those with $1\le a\le \lfloor k/2\rfloor$.
\end{proof}
 
We next identify the values of $k$ for which these defining conditions collapse to a single equation.
\begin{proposition}\label{prop:single-equation-count}
Set $F_{k,p}:=\{a\in\{0,1,\dots,k\}:a\not\le_p k\}.$ If $F_{k,p}$ consists of a single orbit under the involution $a\mapsto k-a$, the block set $\widetilde{\mathcal B}_{W_k^{\mathrm{Luc}}}$ is defined by a single independent equation.
\end{proposition}

\begin{proof}
By Theorem~\ref{thm:Lucas-Cayley}, the block set $\widetilde{\mathcal B}_{W_k^{\mathrm{Luc}}}$ is defined by equations $e_a(T)=0,$ for $a\in F_{k,p}.$
By Lemma~\ref{lem:Lucas-symmetry}, the conditions $e_a(T)=0$ and $e_{k-a}(T)=0$ are equivalent. Hence two indices in the same orbit under $a\mapsto k-a$ give equivalent equations, and thus the block set can be defined by a single independent equation.
\end{proof}

We can now classify the values of $k$ for which $F_{k,p}$ consists of a single orbit under $a\mapsto k-a$.

\begin{theorem}\label{thm:single-equation}
Assume that $3\le k\le q+1$. Then $F_{k,p}$ consists of a single orbit under the involution $a\mapsto k-a$ if and only if one of the following holds:
\begin{enumerate}[label=(\roman*)]
    \item $p=2$ and $k=5$;
    \item $p$ is odd and $k\in\{2p-3,\,2p-2,\,3p-2\}$.
\end{enumerate}
\end{theorem}

\begin{proof}
Since the involution $a\mapsto k-a$ has orbits of size at most $2$, the set $F_{k,p}$ consists of a single orbit if and only if $|F_{k,p}|\in\{1,2\}$. Writing
$k=\sum_{r\ge 0}k_rp^r,$
we have
$|F_{k,p}|=(k+1)-\prod_{r\ge 0}(k_r+1).$

Assume first that $p$ is odd. Let $t$ be the largest index such that $k_t\ne 0$, so that
$k=ap^t+b,$ for $1\le a\le p-1$ and  $0\le b<p^t$.
Write
$$
b=\sum_{r=0}^{t-1}b_rp^r,\qquad A(b):=\prod_{r=0}^{t-1}(b_r+1).
$$
Then $|F_{k,p}|=a(p^t-A(b))+|F_{b,p}|.$
If $t\ge 2$ and $F_{k,p}\ne\varnothing$, then $b$ is not of the form $p^t-1$, hence
$A(b)\le p^{t-1}(p-1),$
so
$p^t-A(b)\ge p^{t-1}.$
Therefore
$$
|F_{k,p}|\ge ap^{t-1}\ge p^{t-1}\ge p\ge 3,
$$
contradicting $|F_{k,p}|\in\{1,2\}$. Thus $t=1$, so
$k=ap+b$, for $1\le a\le p-1$ and $0\le b\le p-1.$
In this case
$|F_{k,p}|=ap+b+1-(a+1)(b+1)=a(p-b-1).$
Requiring $a(p-b-1)\in\{1,2\}$ yields exactly the three possibilities
$$
(a,b)=(1,p-3),\quad (1,p-2),\quad (2,p-2),
$$
that is,
$k\in\{2p-3,\,2p-2,\,3p-2\}.$

Now assume that $p=2$. Write $k=2^t+b$,
where $t$ is the highest nonzero binary digit of $k$ and $0\le b<2^t$. If $t\ge 3$ and $F_{k,2}\ne\varnothing$, then the same argument gives $|F_{k,2}|\ge 2^{t-1}\ge 4,$
again impossible. Hence $t\le 2$. Since $k\ge 3$, the case $t=1$ gives only $k=3$, for which $F_{3,2}=\varnothing$. Thus $t=2$, so
$k=4+b,$ where $0\le b\le 3.$
A direct check gives
$$
|F_{4,2}|=3,\qquad |F_{5,2}|=2,\qquad |F_{6,2}|=3,\qquad |F_{7,2}|=0.
$$
Therefore the only possibility is $k=5$.
\end{proof}

The preceding theorem identifies the cases in which $F_{k,p}$ consists of a single orbit under $a\mapsto k-a$, and in these cases the Cayley description reduces to a single equation.

\begin{corollary}\label{cor:single-equation}
Assume that $3\le k\le q+1$. In each of the following cases, the Cayley description of $\widetilde{\mathcal B}_{W_k^{\mathrm{Luc}}}$ reduces to a single equation:
\begin{enumerate}[label=(\roman*)]
    \item If $p=2$ and $k=5$, then
    $
    \widetilde{\mathcal B}_{W_5^{\mathrm{Luc}}}
    =
    \left\{
    T\in\binom{U_{q+1}}{5}:e_2(T)=0
    \right\}.
    $
    \item If $p$ is odd and $k=2p-3$, then
    $
    \widetilde{\mathcal B}_{W_{2p-3}^{\mathrm{Luc}}}
    =
    \left\{
    T\in\binom{U_{q+1}}{2p-3}:e_{p-2}(T)=0
    \right\}.
    $
    \item If $p$ is odd and $k=2p-2$, then
    $
    \widetilde{\mathcal B}_{W_{2p-2}^{\mathrm{Luc}}}
    =
    \left\{
    T\in\binom{U_{q+1}}{2p-2}:e_{p-1}(T)=0
    \right\}.
    $
    \item If $p$ is odd and $k=3p-2$, then
    $
    \widetilde{\mathcal B}_{W_{3p-2}^{\mathrm{Luc}}}
    =
    \left\{
    T\in\binom{U_{q+1}}{3p-2}:e_{p-1}(T)=0
    \right\}.
    $
\end{enumerate}
\end{corollary}

\begin{proof}
It suffices to determine the forbidden set $F_{k,p}$ in each case. \begin{enumerate}[label=(\roman*)]
    \item If $p=2$ and $k=5$, then $5=(101)_2$, so $F_{5,2}=\{2,3\}.$
By Lemma~\ref{lem:Lucas-symmetry}, the conditions $e_2(T)=0$ and $e_3(T)=0$ are equivalent, so the single independent equation is $e_2(T)=0$.
\item If $p$ is odd and $k=2p-3$, then $k=(1,p-3)_p$, and one checks that
$F_{2p-3,p}=\{p-2,p-1\}.$
These two indices form one orbit under $a\mapsto k-a$, so the single independent equation is $e_{p-2}(T)=0$.
\item If $p$ is odd and $k=2p-2$, then $k=(1,p-2)_p$, and $F_{2p-2,p}=\{p-1\}.$
Thus the single equation is $e_{p-1}(T)=0$.
\item If $p$ is odd and $k=3p-2$, then $k=(2,p-2)_p$, and
$F_{3p-2,p}=\{p-1,2p-1\}.$
These two indices again form one orbit under $a\mapsto k-a$, so the single independent equation is $e_{p-1}(T)=0$.
\end{enumerate}
\end{proof}

We conclude this subsection with two examples. 

\begin{example}[The space \(W_5^{\mathrm{Luc}}\) in characteristic \(2\)]\label{ex:W5-char2}
Assume that $q=2^n$. Since $5=1+2^2$ has binary expansion $(101)_2$, the condition $i\le_2 5$ is equivalent to $i\in\{0,1,4,5\}$. Hence
$$
W_5^{\mathrm{Luc}}=\Span_{\mathbb F_q}\{X^5,X^4Y,XY^4,Y^5\}\subseteq \mathbb F_q[X,Y]_5.
$$
To determine $\mathcal{B}_{W_5^{\mathrm{Luc}}}$,  we first consider blocks containing $\{\infty,0,1\}$.  
Let
$S=\{\infty,0,1,a,b\}$, where $a,b\in\mathbb F_q\setminus\{0,1\}$ and $a\ne b.
$
Then
\begin{align*}
    F_S(X,Y)&=Y\cdot X(X-Y)(X-aY)(X-bY)\\
    &=X^4Y+(1+a+b)X^3Y^2+(a+b+ab)X^2Y^3+abXY^4.
\end{align*}
Therefore
$F_S\in W_5^{\mathrm{Luc}}
$ exactly when $1+a+b=0$ and $a+b+ab=0.$
Eliminating $b$ gives
$$
b=1+a,\qquad a^2+a+1=0.
$$
Thus $\mathcal{B}_{W_5^{\mathrm{Luc}}}\ne\varnothing$ if and only if the polynomial $Z^2+Z+1$ has roots in $\mathbb F_q$, equivalently, if and only if $n$ is even. In that case, if $\omega\in\mathbb F_4\subseteq\mathbb F_q$ satisfies $\omega^2+\omega+1=0$, then the unique block containing $\{\infty,0,1\}$ is
$B_0=\{\infty,0,1,\omega,\omega^2\}$.
Since $\PGL_2(\mathbb F_q)$ acts sharply $3$-transitively on $\mathbb P^1(\mathbb F_q)$ and preserves $\mathcal{B}_{W_5^{\mathrm{Luc}}}$, it follows that every $3$-subset of $\mathbb P^1(\mathbb F_q)$ is contained in exactly one block whenever $\mathcal{B}_{W_5^{\mathrm{Luc}}}\neq\varnothing$.
Hence:
\begin{enumerate}[label=(\roman*)]
    \item if $n$ is odd, then $\mathcal{B}_{W_5^{\mathrm{Luc}}}=\varnothing$;
    \item if $n$ is even, then $(\mathbb P^1(\mathbb F_q),\mathcal{B}_{W_5^{\mathrm{Luc}}})$ is a Steiner system $S(3,5,q+1)$.
\end{enumerate} 
 Since \(5=101_2\), the only integers \(a\) with \(0\le a\le 5\) and \(a\not\le_2 5\) are $ a=2,3.$ By Lemma~\ref{lem:Lucas-symmetry}, we have
    \[
        \widetilde{\mathcal{B}}_{W_5^{\mathrm{Luc}}}
        =
        \left\{
        T\in \binom{U_{q+1}}{5} :
        e_2(T)=0
        \right\}.
    \]
    Thus, when \(n\) is even, $\bigl(U_{q+1},\widetilde{\mathcal{B}}_{W_5^{\mathrm{Luc}}}\bigr)$
    is also a Steiner system \(S(3,5,q+1)\).
    In particular, for \(n\) even, the design parameters are
    \[
        v=q+1,\qquad k=5,\qquad \lambda=1,
    \]
    and consequently
    \[
        b=\frac{\binom{q+1}{3}}{\binom{5}{3}}
        =\frac{(q+1)q(q-1)}{60},
        \qquad
        r=\frac{\binom{q}{2}}{\binom{4}{2}}
        =\frac{q(q-1)}{12}.
    \]
\end{example}
\begin{remark}
    The above Steiner system is exactly the \(3\)-design appearing in Theorem~3 of \cite{Tang21}. While the unit-circle model provides an elegant symmetric condition $e_2(T)=0$, solving it necessitates working over the extension field $\mathbb{F}_{q^2}$ subject to $u^{q+1}=1$. By contrast, the equivalent $\mathbb{P}^1(\mathbb{F}_q)$-model is computationally much more direct, as it allows one to recover the block via straightforward coefficient calculations over the base field $\mathbb{F}_q$.
\end{remark}

\begin{example}[The space \(W_7^{\mathrm{Luc}}\) in characteristic \(3\)]\label{ex:W7-char3}
Assume that $q=3^n$ with $n\ge 2$, so that $7\le q+1$. Since $7=1+2\cdot 3$ has ternary expansion $(21)_3$, the condition $i\le_3 7$ is equivalent to $i\in\{0,1,3,4,6,7\}$. Hence
$$
W_7^{\mathrm{Luc}}=\Span_{\mathbb F_q}\{X^7,X^6Y,X^4Y^3,X^3Y^4,XY^6,Y^7\}\subseteq \mathbb F_q[X,Y]_7.
$$
To describe $\mathcal{B}_{W_7^{\mathrm{Luc}}}$, it is enough to consider blocks containing $\{\infty,0,1\}$. Let
$S=\{\infty,0,1,a,b,c,d\},$
where $a,b,c,d\in\mathbb F_q\setminus\{0,1\}$ are pairwise distinct. Set
$T=\{0,1,a,b,c,d\}$. Then
$$F_S(X,Y)=Y\prod_{t\in T}(X-tY)=\sum_{j=0}^6(-1)^je_j(T)X^{6-j}Y^{j+1}.$$
Since
$W_7^{\mathrm{Luc}}=\Span_{\mathbb F_q}\{X^7,X^6Y,X^4Y^3,X^3Y^4,XY^6,Y^7\},$
it follows that $F_S\in W_7^{\mathrm{Luc}}$ if and only if $e_1(T)=e_4(T)=0.$
Because $0\in T$, this may be rewritten as
$$
1+a+b+c+d=0,\qquad e_4(1,a,b,c,d)=0.
$$
Thus $\mathcal{B}_{W_7^{\mathrm{Luc}}}$ consists precisely of those $7$-subsets $S\subseteq \mathbb P^1(\mathbb F_q)$ whose associated polynomial $F_S$ has vanishing coefficients in the forbidden positions.

Since \(7=21_3\), the integers \(a\) with \(0\le a\le 7\) and \(a\not\le_3 7\) are $a=2,5.$
By Lemma~\ref{lem:Lucas-symmetry}, we have
$$
\widetilde{\mathcal B}_{W_7^{\mathrm{Luc}}}
=
\left\{
T\in\binom{U_{q+1}}{7}:e_2(T)=0
\right\}.
$$
This is exactly the block description appearing in Theorem~2 of \cite{Xu2024}. Whenever $\mathcal{B}_{W_7^{\mathrm{Luc}}}\ne\varnothing$, the previous results imply that
$(U_{q+1},\widetilde{\mathcal B}_{W_7^{\mathrm{Luc}}})$
is a $3$-$(q+1,7,\lambda)$ design for some $\lambda$.
\end{example}
\begin{remark}
    The \(3\)-design obtained here is exactly the one appearing in Theorem~2 of \cite{Xu2024}. In that paper, the authors were not able to determine its parameters explicitly. We shall compute these parameters later. As we will see, working with the \(\mathbb{P}^1(\mathbb{F}_q)\)-model greatly simplifies the calculation.
\end{remark}

\subsection{On the emptiness and nonemptiness of the block sets}
We now study when the associated block sets are empty or nonempty.
We begin with a simple vanishing criterion.
\begin{proposition}\label{prop:p-divides-k}
Assume that $1\le k\le q+1$. If $p\mid k$, then $\mathcal{B}_{W_k^{\mathrm{Luc}}}=\varnothing.$
\end{proposition}

\begin{proof}
Since $p\mid k$, the least base-$p$ digit of $k$ is zero. Hence, if $i\le_p k$, then the least base-$p$ digit of $i$ is also zero, so $p\mid i$. It follows that every monomial $X^{k-i}Y^i$ occurring in $W_k^{\mathrm{Luc}}$ has both exponents divisible by $p$. Therefore every polynomial $f\in W_k^{\mathrm{Luc}}$ can be written in the form
$$f(X,Y)=\sum_j c_j X^{pa_j}Y^{pb_j}.$$
Since $\mathbb F_q$ is perfect, each $c_j$ has a $p$-th root in $\mathbb F_q$, and hence
$
f(X,Y)=g(X,Y)^p
$
for some homogeneous polynomial $g(X,Y)\in \mathbb F_q[X,Y]$. Consequently every zero of $f$ in $\mathbb P^1(\mathbb F_q)$ occurs with multiplicity divisible by $p$. In particular, $f$ cannot vanish at $k$ distinct points of $\mathbb P^1(\mathbb F_q)$. By Proposition~\ref{prop:BW-nonempty}, this implies $\mathcal{B}_{W_k^{\mathrm{Luc}}}=\varnothing$.
\end{proof}

The next lemma provides a useful multiplicative construction.
\begin{lemma}\label{lem:carry-free-product}
Let $\ell_1,\dots,\ell_s$ be nonnegative integers such that $k:=\ell_1+\cdots+\ell_s$
is carry-free in base $p$. If $f_j(X,Y)\in W_{\ell_j}^{\mathrm{Luc}}$ for $1\le j\le s$,
then their product $f_1(X,Y)\cdots f_s(X,Y)\in W_k^{\mathrm{Luc}}.$
\end{lemma}

\begin{proof}
For each $j$, write
$$
f_j(X,Y)=\sum_{i_j\le_p \ell_j} a_{j,i_j}X^{\ell_j-i_j}Y^{i_j}.
$$
Consider a monomial occurring in the product $f_1\cdots f_s$. Its exponent of $Y$ is of the form $i=i_1+\cdots+i_s$ with $i_j\le_p \ell_j$ for all $j$. Write
$$
\ell_j=\sum_{t\ge 0}(\ell_j)_tp^t,\qquad i_j=\sum_{t\ge 0}(i_j)_tp^t,\qquad k=\sum_{t\ge 0}k_tp^t.
$$
Since $i_j\le_p \ell_j$, we have $(i_j)_t\le (\ell_j)_t$
for all $j,t$. 
As the sum $\ell_1+\cdots+\ell_s$ is carry-free, for each $t$ one has
$(\ell_1)_t+\cdots+(\ell_s)_t=k_t<p.$
Hence
$(i_1)_t+\cdots+(i_s)_t\le (\ell_1)_t+\cdots+(\ell_s)_t=k_t<p$ for all $t$. 
Therefore the addition
$i=i_1+\cdots+i_s$ is also carry-free in base $p$, and so
$i_t=(i_1)_t+\cdots+(i_s)_t\le k_t$ for all $t$. 
Thus $i\le_p k$. 
It follows that every monomial occurring in $f_1\cdots f_s$ is of the form $X^{k-i}Y^i$ with $i\le_p k$, and hence
$f_1(X,Y)\cdots f_s(X,Y)\in W_k^{\mathrm{Luc}}.$
\end{proof}

The lemma immediately yields the following reduction principle for nonemptiness.
\begin{corollary}\label{cor:carry-free-block}
Let $\ell_1,\dots,\ell_s$ be nonnegative integers such that $k=\ell_1+\cdots+\ell_s$
is carry-free in base $p$. 
Suppose that $S_j\in \mathcal{B}_{W_{\ell_j}^{\mathrm{Luc}}}$ for $1\le j\le s,$
and that the subsets $S_1,\dots,S_s$ are pairwise disjoint. Then
$$
S:=S_1\cup\cdots\cup S_s
$$
belongs to $\mathcal{B}_{W_k^{\mathrm{Luc}}}$. In particular, $\mathcal{B}_{W_k^{\mathrm{Luc}}}\neq \varnothing.$
\end{corollary}

\begin{proof}
Since the sets $S_1,\dots,S_s$ are pairwise disjoint, we have
$$
F_S(X,Y)=\prod_{j=1}^s F_{S_j}(X,Y).
$$
As $F_{S_j}\in W_{\ell_j}^{\mathrm{Luc}}$ for each $j$, Lemma~\ref{lem:carry-free-product} gives $F_S\in W_k^{\mathrm{Luc}}.$
Hence $S\in \mathcal{B}_{W_k^{\mathrm{Luc}}}$.
\end{proof}

We now record two useful classes of basic blocks.
\begin{proposition}[Multiplicative basic blocks]\label{prop:multiplicative-basic}
Let $d$ be a positive integer such that $d\mid (q-1)$ and $p\nmid d$.
Let $H\le \mathbb F_q^\times$ be a multiplicative subgroup of order $d$. Then
$H\in \mathcal{B}_{W_d^{\mathrm{Luc}}}.$
If, in addition,
$p\nmid (d+1),$ then
$H\cup\{\infty\}\in \mathcal{B}_{W_{d+1}^{\mathrm{Luc}}}.$
\end{proposition}

\begin{proof}
Since $H$ is the subgroup of $d$-th roots of unity in $\mathbb F_q^\times$, we have
$$
F_H(X,Y)=\prod_{a\in H}(X-aY)=X^d-Y^d.
$$
Only the coefficients in positions $0$ and $d$ occur, and these are always allowed. Hence $H\in \mathcal{B}_{W_d^{\mathrm{Luc}}}$.
Now set $S':=H\cup\{\infty\}$. Then
$$
F_{S'}(X,Y)=Y(X^d-Y^d).
$$
The only exponents of $Y$ occurring here are $1$ and $d+1$.
The exponent $d+1$ is always allowed, while $1\le_p d+1$
holds exactly when the least base-$p$ digit of $d+1$ is nonzero, equivalently when $p\nmid (d+1)$. Therefore, under this additional assumption, $H\cup\{\infty\}\in \mathcal{B}_{W_{d+1}^{\mathrm{Luc}}}$.
\end{proof}

\begin{proposition}[Subfield-line basic blocks]\label{prop:subfield-line-basic}
Let $q=p^e$, and let $m$ be a positive integer such that $m\mid e$. Then the projective subline
$$
S=\mathbb P^1(\mathbb F_{p^m})=\mathbb F_{p^m}\cup\{\infty\}\subseteq \mathbb P^1(\mathbb F_q)
$$
satisfies $S\in \mathcal{B}_{W_{p^m+1}^{\mathrm{Luc}}}.$
In particular, $\mathcal{B}_{W_{p^m+1}^{\mathrm{Luc}}}\neq \varnothing.$
\end{proposition}

\begin{proof}
The polynomial vanishing exactly on $\mathbb F_{p^m}$ is $\prod_{a\in\mathbb F_{p^m}}(X-aY) = X^{p^m}-XY^{p^m-1}$.
Therefore, the associated polynomial for $S = \mathbb F_{p^m}\cup\{\infty\}$ is
$$
F_S(X,Y) = Y(X^{p^m}-XY^{p^m-1}) = X^{p^m}Y - XY^{p^m}.
$$
The exponents of $Y$ occurring in $F_S(X,Y)$ are exactly $1$ and $p^m$. For $k = p^m+1$, its base-$p$ expansion simply consists of $1$'s at the $m$-th and $0$-th positions. It is then immediate that $1 \le_p k$ and $p^m \le_p k$. Thus, $F_S \in W_{p^m+1}^{\mathrm{Luc}}$, which yields $S\in \mathcal{B}_{W_{p^m+1}^{\mathrm{Luc}}}$.
\end{proof}

The preceding proposition completely determines the case $k=p^m+1$.
\begin{proposition}\label{prop:q0-plus-one}
Let $q=p^e$, and let $q_0=p^m$ with $1\le m\le e$. Then
$$
|\mathcal{B}_{W_{q_0+1}^{\mathrm{Luc}}}|=
\begin{cases}
\dfrac{\binom{q+1}{3}}{\binom{q_0+1}{3}}
=\dfrac{q(q^2-1)}{q_0(q_0^2-1)}, & \text{if }m\mid e,\\[8pt]
0, & \text{if }m\nmid e.
\end{cases}
$$
Moreover, if $m\mid e$, then $(\mathbb P^1(\mathbb F_q),\mathcal{B}_{W_{q_0+1}^{\mathrm{Luc}}})$
is a Steiner system $S(3,q_0+1,q+1),$ that is, a $3$-$(q+1,q_0+1,1)$ design.
\end{proposition}

\begin{proof}
If $m\mid e$, then Proposition~\ref{prop:subfield-line-basic} shows that
$\mathcal{B}_{W_{q_0+1}^{\mathrm{Luc}}}\neq \varnothing.$
Conversely, suppose that
$\mathcal{B}_{W_{q_0+1}^{\mathrm{Luc}}}\neq \varnothing.$
Since $\PGL_2(\mathbb F_q)$ acts sharply $3$-transitively on $\mathbb P^1(\mathbb F_q)$ and preserves $\mathcal{B}_{W_{q_0+1}^{\mathrm{Luc}}}$, it suffices to study the blocks containing $\{\infty,0,1\}$. Let $S\in \mathcal{B}_{W_{q_0+1}^{\mathrm{Luc}}}$ contain $\{\infty,0,1\}$, and write
$S=\{\infty\}\cup T,$ where $T\subseteq \mathbb F_q$ has cardinality $q_0$ and contains $0$ and $1$. Define $P_T(Z):=\prod_{t\in T}(Z-t)\in \mathbb F_q[Z].$
Then
$$
F_S(X,Y)=Y\prod_{t\in T}(X-tY)=Y^{q_0+1}P_T(X/Y).
$$
Since
$W_{q_0+1}^{\mathrm{Luc}}=\Span_{\mathbb F_q}\{X^{q_0+1},X^{q_0}Y,XY^{q_0},Y^{q_0+1}\},$
and every term of $F_S(X,Y)$ contains a factor $Y$, it follows that
$P_T(Z)=Z^{q_0}+cZ+d$
for some $c,d\in \mathbb F_q$. As $0\in T$, we have $d=0$. As $1\in T$, we have
$0=P_T(1)=1+c,$ so $c=-1$. Therefore $P_T(Z)=Z^{q_0}-Z.$
Thus $T$ is precisely the set of roots in $\mathbb F_q$ of the polynomial $Z^{q_0}-Z$. The roots of $Z^{p^m}-Z$ in $\mathbb F_q$ form the unique subfield of $\mathbb F_q$ of order $p^{\gcd(m,e)}$, namely $\mathbb F_{p^{\gcd(m,e)}}$. Since $|T|=q_0=p^m$, we must have $\gcd(m,e)=m$, that is, $m\mid e.$
We have shown that
$\mathcal{B}_{W_{q_0+1}^{\mathrm{Luc}}}\neq \varnothing$ if and only if $ m\mid e.$

Assume now that $m\mid e$. Then the argument above shows that the unique block containing $\{\infty,0,1\}$ is $\mathbb F_{q_0}\cup\{\infty\}.$
By the sharp $3$-transitivity of $\PGL_2(\mathbb F_q)$, every $3$-subset of $\mathbb P^1(\mathbb F_q)$ is therefore contained in exactly one block of $\mathcal{B}_{W_{q_0+1}^{\mathrm{Luc}}}$. 
Hence
$(\mathbb P^1(\mathbb F_q),\mathcal{B}_{W_{q_0+1}^{\mathrm{Luc}}})$
is a Steiner system $S(3,q_0+1,q+1)$, equivalently a $3$-$(q+1,q_0+1,1)$ design. 
Finally,
$$
|\mathcal{B}_{W_{q_0+1}^{\mathrm{Luc}}}|=\frac{\binom{q+1}{3}}{\binom{q_0+1}{3}}
=\frac{q(q^2-1)}{q_0(q_0^2-1)}.
$$
This completes the proof.
\end{proof}
%\subsection{The Ternary Melas Code and The Size of $\mathcal{B}_{W_7^{Luc}}$}
\section{\texorpdfstring{Determining parameters for the designs associated with $W_7^{\mathrm{Luc}}$}{Determining parameters for the designs associated with W7Luc}}\label{section:Lucas2}
In this section, using the known weight distribution of the ternary Melas code, we determine the parameters of the designs associated with $W_7^{\mathrm{Luc}}$. We first compute the parameter $\lambda_2$ of the $3$-$(q+1,7,\lambda_2)$ design introduced in Example~\ref{ex:W7-char3}, thereby resolving the case left open in \cite[Theorems~2 and~5]{Xu2024}. We then determine the remaining parameter $\lambda_1$ from \cite[Theorems~1 and~4]{Xu2024}.

Let $q=3^m$. The \emph{ternary Melas code} $M(q)$ is the $\mathbb F_3$-linear code of length $q-1$, defined by
$$
M(q):=
\left\{
(c_x)_{x\in\mathbb F_q^\times}\in \mathbb F_3^{\,q-1}:
\sum_{x\in\mathbb F_q^\times}c_xx=0,\ 
\sum_{x\in\mathbb F_q^\times}c_xx^{-1}=0
\right\}.
$$

\begin{lemma}\label{lem:Melas-weight-5}
Let $q=3^m$. Then
$$
A_3=0,\qquad
A_5=\frac{4(q-1)\bigl(q^2+(( -1)^m-14)q+36\bigr)}{15},
$$
where $A_i$ denotes the number of codewords of weight $i$ in $M(q)$ for $i\ge0$.
\end{lemma}

\begin{proof}
This gives the weight-$3$ and weight-$5$ terms in the known weight distribution of the ternary Melas code; see \cite[Table 6.1]{van1992weight}.
\end{proof}

We now compute the parameter of the $3$-design in Example~\ref{ex:W7-char3}.

\begin{theorem}\label{thm:lambda2}
Let $q=3^m$, and let $\lambda_2$ denote the number of blocks in $\mathcal{B}_{W_7^{\mathrm{Luc}}}$ containing any fixed $3$-subset of $\mathbb P^1(\mathbb F_q)$. Then
$$
\lambda_2=\frac{q^2+(( -1)^m-14)q+36}{24}.
$$
It follows that,
$$
\bigl|\mathcal{B}_{W_7^{\mathrm{Luc}}}\bigr|
=
\frac{q(q^2-1)\bigl(q^2+(( -1)^m-14)q+36\bigr)}{5040}.
$$
In particular,
$ \mathcal{B}_{W_7^{\mathrm{Luc}}}\neq\varnothing$
if and only if  $\lambda_2>0$ (if and only if $m\ge 3$);
in that case
$\bigl(\mathbb P^1(\mathbb F_q),\mathcal{B}_{W_7^{\mathrm{Luc}}}\bigr)$ is a $3$-$(q+1,7,\lambda_2)$ design.
\end{theorem}

\begin{proof}
Fix the three points
$T_0:=\{\infty,0,-1\}\subseteq \mathbb P^1(\mathbb F_q).$
By definition,
$\lambda_2=\#\{\,B\in \mathcal{B}_{W_7^{\mathrm{Luc}}}:T_0\subseteq B\,\}.$
Every such block has the form
$$
B=\{\infty,0,-1,x_1,x_2,x_3,x_4\},
$$
where $x_1,x_2,x_3,x_4\in\mathbb F_q^\times\setminus\{-1\}$ are pairwise distinct. 
For $1\le i\le 4$, set
$e_i=\sigma_i(x_1,x_2,x_3,x_4).$
Then
$$
\prod_{j=1}^4(X-x_jY)=X^4-e_1X^3Y+e_2X^2Y^2-e_3XY^3+e_4Y^4,
$$
and hence
\begin{align*}
F_B(X,Y)
&=Y\cdot X\cdot (X+Y)\prod_{j=1}^4(X-x_jY) \\
&=X^6Y+(1-e_1)X^5Y^2+(e_2-e_1)X^4Y^3 
 +(e_2-e_3)X^3Y^4+(e_4-e_3)X^2Y^5+e_4XY^6.
\end{align*}
Since
$$
W_7^{\mathrm{Luc}}=\Span_{\mathbb F_q}\{X^7,X^6Y,X^4Y^3,X^3Y^4,XY^6,Y^7\},
$$
we obtain $B\in \mathcal{B}_{W_7^{\mathrm{Luc}}}$ if and only if $1-e_1=0,\ e_4-e_3=0,$ that is,
$e_1=1,\ e_3=e_4.$
Since
$$
x_1+x_2+x_3+x_4=e_1,\qquad
x_1^{-1}+x_2^{-1}+x_3^{-1}+x_4^{-1}=\frac{e_3}{e_4},
$$
it follows that $B\in \mathcal{B}_{W_7^{\mathrm{Luc}}}$
if and only if 
$$x_1+x_2+x_3+x_4=1,\ 
x_1^{-1}+x_2^{-1}+x_3^{-1}+x_4^{-1}=1.$$
Let
$
\Lambda:=
\Bigl\{
\{x_1,x_2,x_3,x_4\}\subseteq \mathbb F_q^\times\setminus\{-1\}:
x_i\ \text{pairwise distinct},
\ \sum_{i=1}^4x_i=1,\ \sum_{i=1}^4x_i^{-1}=1
\Bigr\}.
$
Then $\lambda_2=|\Lambda|.$
Define
$ P:=\{(c,t):c\in M(q),\ \mathrm{wt}(c)=5,\ t\in \Supp(c)\}.
$
Since every weight-$5$ codeword has exactly five support positions,
$|P|=5A_5.$
For $(c,t)\in P$, write
$$
c=\sum_{x\in\Supp(c)}\varepsilon_xe_x,\qquad \varepsilon_x\in\{\pm 1\}\subseteq\mathbb F_3.
$$
For each $x\in \Supp(c)\setminus\{t\}$, define
$u_x:=-\varepsilon_x\varepsilon_t\,\frac{x}{t}.$
Set $\Phi(c,t):=\{u_x:x\in \Supp(c)\setminus\{t\}\}.$
We claim that $\Phi$ defines a map
$\Phi:P\longrightarrow \Lambda.$
Since $c\in M(q)$, we have
$\sum_{x\in\Supp(c)}\varepsilon_xx=0,$ and 
$\sum_{x\in\Supp(c)}\varepsilon_xx^{-1}=0.$
Separating the term indexed by $t$ and dividing by $-\varepsilon_tt$ and $-\varepsilon_tt^{-1}$, respectively, we obtain
$$
\sum_{x\in\Supp(c)\setminus\{t\}}u_x=1,
\qquad
\sum_{x\in\Supp(c)\setminus\{t\}}u_x^{-1}=1.
$$
Thus only distinctness and the exclusion of the value $-1$ remain to be checked.

First, no $u_x$ can be equal to $1$. Indeed, if $u_x=1$, then
$-\varepsilon_x\varepsilon_t\,\frac{x}{t}=1,$
so
$\varepsilon_xx=-\varepsilon_tt,$ and $\varepsilon_xx^{-1}=-\varepsilon_tt^{-1}.$
Deleting the coordinates $x$ and $t$ would then produce a codeword of weight $3$ in $M(q)$, contradicting $A_3=0$.

Next, no $u_x$ can be equal to $-1$. Suppose $u_x=-1$ for some $x\ne t$, and let the remaining normalized values be $v_1,v_2,v_3$. Then
$$
v_1+v_2+v_3=-1,\qquad v_1^{-1}+v_2^{-1}+v_3^{-1}=-1.
$$
If
$$
e_1''=v_1+v_2+v_3,\qquad e_2''=v_1v_2+v_1v_3+v_2v_3,\qquad e_3''=v_1v_2v_3,
$$
then $e_1''=-1$ and $e_2''/e_3''=-1$, so $e_2''=-e_3''$. Hence $v_1,v_2,v_3$ are the roots of
$$
T^3-e_1''T^2+e_2''T-e_3''=T^3+T^2+e_2''T+e_2''=(T+1)(T^2+e_2'').
$$
Thus one of the $v_i$ equals $-1$. Repeating the same argument gives at least two normalized values equal to $-1$, which is impossible, since $u_x=-1$ is equivalent to $x=\varepsilon_x\varepsilon_tt$, and for fixed $t$ there is at most one such $x\neq t$.
Finally, the values $u_x$ are pairwise distinct. Suppose $u_x=u_y$ for distinct $x,y\in \Supp(c)\setminus\{t\}$. Excluding the case $u_x=-1$ already treated, write $u_x=u_y=u\neq -1$, and let the remaining two normalized values be $v,w$. Then
$$
u+u+v+w=1,\qquad u^{-1}+u^{-1}+v^{-1}+w^{-1}=1.
$$
Since the characteristic is $3$, this becomes
$$
v+w=1+u,\qquad v^{-1}+w^{-1}=1+u^{-1}=\frac{u+1}{u}.
$$
Hence
$$
\frac{v+w}{vw}=\frac{u+1}{u},
$$
so, because $u\neq -1$, we obtain $vw=u.$
Therefore $v$ and $w$ are the roots of
$$
T^2-(1+u)T+u=(T-1)(T-u).
$$
Thus one of $v,w$ equals $1$, contradicting the fact already proved that no normalized value equals $1$.
We have shown that $\Phi(c,t)\in \Lambda$, so $\Phi$ is well defined.

We now compute the cardinality of each fiber. 
Fix $U=\{u_1,u_2,u_3,u_4\}\in \Lambda.$
We first show that $1\notin U$. 
Suppose, to the contrary, that $U=\{1,v_1,v_2,v_3\}$.
Since $U\in\Lambda$, we have
\[
v_1+v_2+v_3=0,
\qquad
v_1^{-1}+v_2^{-1}+v_3^{-1}=0.
\]
Let
\[
e_1''=v_1+v_2+v_3,\qquad
e_2''=v_1v_2+v_1v_3+v_2v_3,\qquad
e_3''=v_1v_2v_3.
\]
Then $e_1''=0$ and $e_2''/e_3''=0$, so $e_2''=0$. 
Hence $v_1,v_2,v_3$ are the three roots of $T^3-e_3''$.
Since $q=3^m$, the Frobenius map $x\mapsto x^3$ is a bijection on $\mathbb F_q$, so the polynomial $T^3-e_3''$ has a unique root in $\mathbb F_q$, counted with multiplicity. This contradicts the fact that $v_1,v_2,v_3$ are pairwise distinct. Therefore $1\notin U$.
Next we show that $U$ contains no pair $\{u,-u\}$. 
Suppose, to the contrary, that $U=\{u,-u,v,w\}$. 
Since $U\in\Lambda$, we have  $v+w=1$ and $ v^{-1}+w^{-1}=1$.
Thus $\frac{v+w}{vw}=1$.
Because $v+w=1$, it follows that $vw=1$. 
Hence $v$ and $w$ are the two roots of
\[
T^2-(v+w)T+vw=T^2-T+1=(T+1)^2
\]
in characteristic $3$. Therefore $v=w=-1$, contradicting both the distinctness of the elements of $U$ and the condition
$U\subseteq \mathbb F_q^\times\setminus\{-1\}$.
So $U$ contains no pair $\{u,-u\}$.

Choose $t\in\mathbb F_q^\times,$ which gives $q-1$ choices, choose $\varepsilon_t\in\{\pm 1\},$
which gives $2$ choices, and choose arbitrary sign $\delta_1,\delta_2,\delta_3,\delta_4\in\{\pm 1\},$
which gives $2^4=16$ choices.  For $1\le i\le 4$, define
$x_i:=\delta_i u_it$ and $ \varepsilon_{x_i}:=-\delta_i\varepsilon_t$. Set
$c:=\varepsilon_te_t+\sum_{i=1}^4\varepsilon_{x_i}e_{x_i}.$
We claim that the five positions $t,x_1,x_2,x_3,x_4$ are pairwise distinct. 
First, if $x_i=t$, then $\delta_i u_i=1$, so $u_i=\delta_i\in\{\pm1\}$. 
Since $U\subseteq \mathbb F_q^\times\setminus\{-1\}$ and we have already proved that $1\notin U$, this is impossible. 
Next, if $x_i=x_j$ for some $i\ne j$, then $\delta_i u_i=\delta_j u_j$, hence $u_i=\delta_i\delta_j\,u_j$.
If $\delta_i\delta_j=1$, then $u_i=u_j$, contradicting the distinctness of the elements of $U$. If $\delta_i\delta_j=-1$, then $u_i=-u_j$, contradicting the fact just proved that $U$ contains no pair $\{u,-u\}$. Therefore $t,x_1,x_2,x_3,x_4$ are pairwise distinct, and so $\wt(c)=5$.
Exactly as in the construction above, one checks that $c\in M(q)$, $\mathrm{wt}(c)=5$, and $\Phi(c,t)=U.$
Conversely, every preimage of $U$ arises in this way: if $(c,t)\in\Phi^{-1}(U)$ and
\[
c=\varepsilon_t e_t+\sum_{i=1}^4\varepsilon_{x_i}e_{x_i},
\]
then, after indexing the four elements of $\Supp(c)\setminus\{t\}$ so that
$u_i=-\varepsilon_{x_i}\varepsilon_t\frac{x_i}{t}$ for $1\le i\le 4$,
we recover
$x_i=\delta_i u_i t$ and $ \varepsilon_{x_i}=-\delta_i\varepsilon_t$,
with $\delta_i:=-\varepsilon_{x_i}\varepsilon_t\in\{\pm1\}$.
Thus every element of $\Phi^{-1}(U)$ is obtained uniquely from a choice of
\[
t\in\mathbb F_q^\times,\qquad \varepsilon_t\in\{\pm1\},\qquad
(\delta_1,\delta_2,\delta_3,\delta_4)\in\{\pm1\}^4.
\]
Therefore
$$
|\Phi^{-1}(U)|=(q-1)\cdot 2\cdot 16=32(q-1)
$$
for every $U\in \Lambda$.
Counting $P$ in two ways gives
$$
5A_5=|P|=\sum_{U\in\Lambda}|\Phi^{-1}(U)|=32(q-1)|\Lambda|=32(q-1)\lambda_2,
$$
and hence
$\lambda_2=\frac{5A_5}{32(q-1)}.$
Applying Lemma~\ref{lem:Melas-weight-5}, we obtain
$$
\lambda_2=\frac{q^2+(( -1)^m-14)q+36}{24}.
$$
The formula for $|\mathcal{B}_{W_7^{\mathrm{Luc}}}|$ follows from
$$
|\mathcal{B}_{W_7^{\mathrm{Luc}}}|=\lambda_2\frac{\binom{q+1}{3}}{\binom{7}{3}}.
$$
The final assertion follows from Proposition~\ref{prop:BW-design}.
\end{proof}
\begin{remark}
    As an immediate consequence, the parameter \(\lambda_2\) appearing in Theorems~2 and~5 of \cite{Xu2024} is now determined explicitly.
\end{remark}

To compute the parameter $\lambda_1$ from \cite[Theorems~1 and~4]{Xu2024}, we first introduce an important lemma.

\begin{lemma}\label{lem:unit-circle-aux}
For any pairwise distinct elements $y_1,\dots,y_5\in U_{q+1}\setminus\{1\}$, one has
$$
e_2(y_1,\dots,y_5,1,1)\neq 0.
$$
\end{lemma}

\begin{proof}
Suppose, to the contrary, that
$e_2(y_1,\dots,y_5,1,1)=0.$
Let
$$
g(X)=\prod_{i=1}^5(X-y_i)=X^5+a_4X^4+a_3X^3+a_2X^2+a_1X+a_0.
$$
Since $g(X)(X-1)^2=\prod_{i=1}^5(X-y_i)(X-1)^2,$
the roots of $g(X)(X-1)^2$ are precisely $y_1,\dots,y_5,1,1$. Hence the coefficient of $X^5$ in this polynomial is $e_2(y_1,\dots,y_5,1,1),$ and therefore vanishes by assumption.
Since the characteristic is $3$, we have $(X-1)^2=X^2+X+1.$
Thus
$$
g(X)(X-1)^2=g(X)(X^2+X+1),
$$
and the coefficient of $X^5$ in this product is $1+a_4+a_3$, so $1+a_4+a_3=0.$

Write
$$
g(X)=X^5-e_1X^4+e_2X^3-e_3X^2+e_4X-e_5,
$$
where $e_j=e_j(y_1,\dots,y_5)$. Since each $y_i\in U_{q+1}$, it satisfies $y_i^q=y_i^{-1}$. We have
$$
e_j(y_1,\dots,y_5)^q=e_j(y_1^{-1},\dots,y_5^{-1})=\frac{e_{5-j}(y_1,\dots,y_5)}{e_5(y_1,\dots,y_5)}
$$
for $0\le j\le 5$. Therefore $e_4=e_5e_1^q,$ and $ e_3=e_5e_2^q.$

Since
$$
a_4=-e_1,\quad a_3=e_2,\quad a_2=-e_3,\quad a_1=e_4,\quad a_0=-e_5,
$$
it follows that $a_1=a_0a_4^q$ and $ a_2=a_0a_3^q.$
Taking $q$-th powers in $1+a_4+a_3=0$, we obtain $1+a_4^q+a_3^q=0.$
Hence
\begin{align*}
g(1)
&=1+a_4+a_3+a_2+a_1+a_0 \\
&=(1+a_4+a_3)+a_0(1+a_4^q+a_3^q)=0.
\end{align*}
Thus $1$ is a root of $g(X)$, which means that $y_i=1$ for some $i$. This contradicts the assumption that $y_1,\dots,y_5\in U_{q+1}\setminus\{1\}$. Therefore
$e_2(y_1,\dots,y_5,1,1)\neq 0.$
\end{proof}

We next compute the parameter $\lambda_1$ from \cite[Theorems~1 and~4]{Xu2024} by relating the $7$-subset design in Theorem~\ref{thm:lambda2} to the $6$-subset design appearing in \cite[Theorems~1 and~4]{Xu2024}.

% Also recall that, by Example~\ref{ex:W7-char3},
% $$
% \widetilde{\mathcal B}_{W_7^{\mathrm{Luc}}}
% =
% \left\{
% B\in\binom{U_{q+1}}{7}:e_2(B)=0
% \right\}.
% $$

\begin{proposition}\label{pro:lambda1}Let $q=3^m$. For a $6$-subset $A=\{x_1,\dots,x_6\}\subseteq U_{q+1}$, write $e_i(A)=e_i(x_1,\dots,x_6),$ where $0\le i\le 6$.
    Let $\mathcal A=
\left\{
A\in\binom{U_{q+1}}{6}: e_4(A)e_2(A)=e_5(A)e_1(A)
\right\}$ and  $
 \widetilde{\mathcal B}_{W_7^{\mathrm{Luc}}}
 =
 \left\{
 B\in\binom{U_{q+1}}{7}:e_2(B)=0
 \right\}.
 $ Then $|\mathcal A|=7|\widetilde{\mathcal B}_{W_7^{\mathrm{Luc}}} |$.
\end{proposition}
\begin{proof}
Set $\mathcal B^*:=\{(B,x):B\in \widetilde{\mathcal B}_{W_7^{\mathrm{Luc}}},\ x\in B\}.$
It follows that  $|\mathcal B^*|=7\,|\widetilde{\mathcal B}_{W_7^{\mathrm{Luc}}}|.$
We claim that the map
$$
\phi:\mathcal B^*\longrightarrow \mathcal A,\qquad (B,x)\longmapsto B\setminus\{x\},
$$
is a bijection.
Let $(B,x)\in \mathcal B^*$, and set
$A:=B\setminus\{x\}.$
Since $e_2(B)=0$, we have
$e_2(A)+xe_1(A)=0.$
Hence $e_2(A)^{q+1}=(-1)^{q+1}x^{q+1}e_1(A)^{q+1}=e_1(A)^{q+1},$ because $q+1$ is even and  $x\in U_{q+1}$ implies $x^{q+1}=1$.
Using the symmetry relation $e_{6-i}(A)=e_6(A)e_i(A)^q,$ it is equivalent to $e_4(A)e_2(A)=e_5(A)e_1(A).$
Thus $A\in \mathcal A$, so $\phi$ is well defined.
Conversely, let
$A=\{x_1,\dots,x_6\}\in \mathcal A.$
We first show that $e_1(A)\neq 0$. Otherwise the defining equation for $\mathcal A$ gives
$e_2(A)^{q+1}=e_1(A)^{q+1}=0,$
hence $e_2(A)=0$. Dividing by $x_6$, we obtain
$$
e_1\!\left(\frac{x_1}{x_6},\dots,\frac{x_5}{x_6},1\right)=0,
\qquad
e_2\!\left(\frac{x_1}{x_6},\dots,\frac{x_5}{x_6},1\right)=0.
$$
Therefore
$$
e_2\!\left(\frac{x_1}{x_6},\dots,\frac{x_5}{x_6},1,1\right)=e_2\!\left(\frac{x_1}{x_6},\dots,\frac{x_5}{x_6},1\right)+1\cdot e_1\!\left(\frac{x_1}{x_6},\dots,\frac{x_5}{x_6},1\right)=0
$$
contradicting Lemma~\ref{lem:unit-circle-aux}. Thus $e_1(A)\neq 0$.

Now define
$x:=-\frac{e_2(A)}{e_1(A)}.$
Since $A\in \mathcal A$, we have $x^{q+1}=1$, so $x\in U_{q+1}$. We next show that $x\notin A$. Suppose, for instance, that $x=x_1$. Then $e_2(x_1,x_1,x_2,\dots,x_6)=0.$
Dividing by $x_1$ gives $e_2\!\left(1,1,\frac{x_2}{x_1},\dots,\frac{x_6}{x_1}\right)=0,$
again contradicting Lemma~\ref{lem:unit-circle-aux}. Hence $x\notin A$.
Set
$B:=A\cup\{x\}.$
Then $e_2(B)=e_2(A)+xe_1(A)=0,$
so
$B\in \widetilde{\mathcal B}_{W_7^{\mathrm{Luc}}}$
and clearly $\phi(B,x)=A.$
Thus $\phi$ is surjective. Finally, we show that $\phi$ is injective. If $\phi(B_1,x_1)=\phi(B_2,x_2)=A$, then we have  $e_2(B_1)= e_2(A)+x_1e_1(A)=0$ and  $e_2(B_2)= e_2(A)+x_2e_1(A)=0$. Since $e_1(A)\neq0$, we have $x_1=x_2$ and then $B_1=B_2$.

Therefore $\phi$ is a bijection, and hence
$$
|\mathcal A|=|\mathcal B^*|=7\,|\widetilde{\mathcal B}_{W_7^{\mathrm{Luc}}}|.
$$
This completes the proof.
\end{proof}
\begin{theorem}
 Let $\lambda_1$ denote the parameter of the $3$-$(q+1,6,\lambda_1)$ design associated with the block set $\mathcal A$, equivalently, the parameter occurring in \cite[Theorems~1 and~4]{Xu2024}. Then
$$
\lambda_1=4\lambda_2=\frac{q^2+(( -1)^m-14)q+36}{6}.
$$
In particular, the parameter $\lambda_1$ left open in \cite[Theorems~1 and~4]{Xu2024} is determined explicitly.
\end{theorem}
\begin{proof}
By (\ref{solve_b}), we have 
$$
|\mathcal A|=\lambda_1\frac{\binom{q+1}{3}}{\binom{6}{3}},
\qquad
|\widetilde{\mathcal B}_{W_7^{\mathrm{Luc}}}|=\lambda_2\frac{\binom{q+1}{3}}{\binom{7}{3}}.
$$
Using $|\mathcal A|=7|\widetilde{\mathcal B}_{W_7^{\mathrm{Luc}}}|$ in  Proposition~\ref{pro:lambda1}, we obtain
$$
\lambda_1\frac{\binom{q+1}{3}}{20}
=
7\lambda_2\frac{\binom{q+1}{3}}{35},
$$
and therefore
$\lambda_1=4\lambda_2.$
Substituting the value of $\lambda_2$ from Theorem~\ref{thm:lambda2}, we get
$\lambda_1=\frac{q^2+(( -1)^m-14)q+36}{6}.$   
\end{proof}

\section{Conclusion}\label{section:Conclusion}

In this paper, we have established a unified framework for constructing $3$-designs from $\GL_2(\mathbb F_q)$-invariant subspaces of $\mathbb F_q[X,Y]_k$. The point of view adopted here links invariant subspaces, block families on $\mathbb P^1(\mathbb F_q)$, and associated subcodes of the projective Reed--Solomon code. In particular, when $k\le q$, the framework also yields $3$-designs from the supports of minimum-weight codewords in the associated subcodes and from the supports of suitable fixed-weight codewords in their duals. The Cayley transform provides an equivalent formulation on the unit circle, where the block conditions become explicit linear relations among elementary symmetric polynomials.

For the Lucas subspaces $W_k^{\mathrm{Luc}}$, this framework leads to concrete descriptions of the corresponding block sets and to several criteria for emptiness and nonemptiness. It also produces explicit families of blocks, including those arising from subfield lines, and completely determines the case $k=p^m+1$, where the resulting design is the Steiner system $S(3,p^m+1,q+1)$ whenever it exists. In the ternary case $k=7$, the known weight distribution of the ternary Melas code further allows us to determine explicitly the parameters left open in \cite{Xu2024}.

Several natural questions remain for further investigation. One direction is to study other families of $\GL_2(\mathbb F_q)$-invariant subspaces, to determine when the associated block families are nonempty, and to analyze the support designs arising from the corresponding codes and dual codes. It would also be interesting to obtain more explicit criteria for emptiness and nonemptiness, as well as closed formulas for the parameters of the resulting designs. Finally, determining the automorphism groups and exploring further coding-theoretic properties of the associated codes remain worthwhile problems.
%\section*{Declaration of generative AI and AI-assisted technologies in the manuscript preparation process}
%During the preparation of this work, the authors used Gemini to improve the language and readability of the manuscript. After using this tool, the authors reviewed and edited the content as needed and take full responsibility for the content of the published article.
\bibliographystyle{plain}
\bibliography{ref}
\end{document}